\documentclass{amsart}
\usepackage{graphicx,euscript,amsmath}
\usepackage[alphabetic]{amsrefs}

\newtheorem{theorem}{Theorem}[section]

\newtheorem*{lemma*}{Lemma}
\theoremstyle{definition}
\newtheorem{remark}{Remark}
\newtheorem*{example*}{Example}

\def\vc#1{\mathbf #1}
\def\mymatrix#1{\begin{bmatrix}#1\end{bmatrix}}

\def\O{{\mathcal O}}
\def\F{{\mathcal F}}
\long\def\ignore#1{}
\def\tt{\theta}

\def\cc{\gamma}

\def\H{{\mathcal H}}

\def\ldot{\odot}

\def\nd{\noindent}
\def\fref#1{Figure~\ref{#1}}

\def\ss{\sigma}

 \providecommand{\Pic}{\mathop{\rm Pic}\nolimits}

\def\Bbb#1{{\mathbb #1}}

\def\E{{\mathcal E}}

\def\dd{\delta}

\def\CC{{\Gamma}}

\def\LL{\Lambda}

\def\K{{\mathcal K}}

\def\L{{\mathcal L}}

\def\proj{{\rm{proj}}}

\def\ll{\lambda}
\def\H{{\mathcal H}}
\def\odot{\circ}
\def\P{{\mathcal P}}

\begin{document}

\title{A game of packings}
\author{Arthur Baragar}
\author{Daniel Lautzenheiser}
\begin{abstract}
In this note, we investigate an infinite one parameter family of circle packings, each with a set of three mutually tangent circles.  We use these to generate an infinite set of circle packings with the {\em Apollonian property}.  That is, every circle in the packing is a member of a cluster of four mutually tangent circles.   
\end{abstract}
\address{Department of Mathematical Sciences, University of Nevada, Las Vegas, NV 89154-4020}
\email{baragar@unlv.nevada.edu}
\address{Eastern Sierra College Center,
4090 W. Line Street,
Bishop, CA 93514-7306}
\email{daniel.lautzenheiser@cerrocoso.edu}
\thanks{\nd \LaTeX ed \today.}

\maketitle

Let us play a game.  Suppose we are given a circle, and inside that circle we have three mutually tangent discs, two of which are tangent to the outside circle.  The challenge is to fill in the remaining space with discs in a logical and symmetric way, as is done in \fref{fig1}.  The rules are vague, but will become clearer as we learn more about the game.     
\begin{figure}
\begin{center}
\includegraphics[width=120pt]{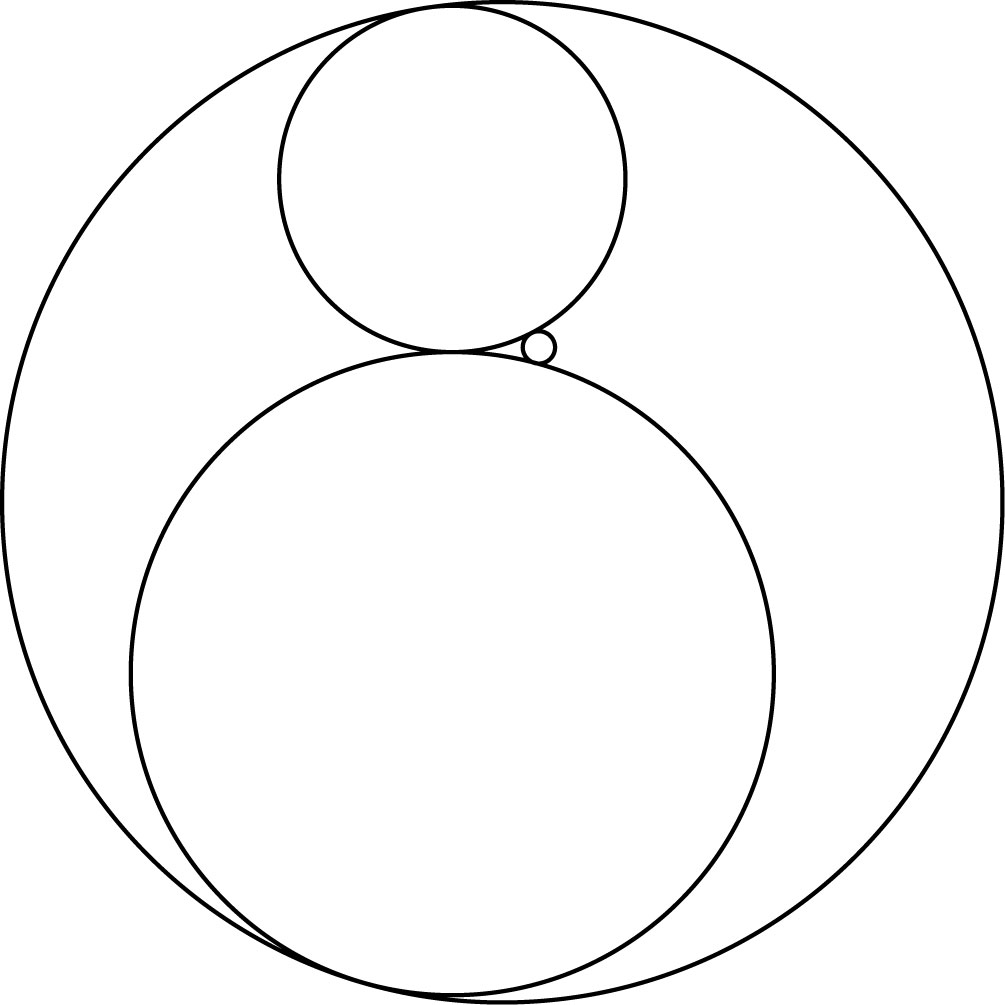}\hspace{.1in}\includegraphics{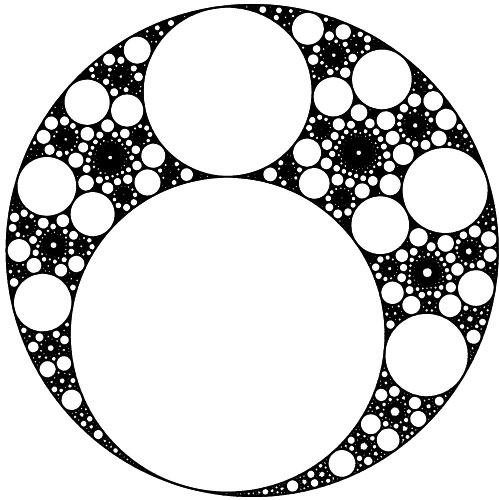}
\end{center}
\caption{\label{fig1}The challenge:  Given an initial configuration, fill in the remaining space so as to have a nice circle packing.}
\end{figure}
The initial setup has, modulo inversion, a one dimensional degree of freedom.  To see this, we invert in the point of tangency of the two discs that are tangent to the original circle.  This gives us the strip version in \fref{fig2}.  
\begin{figure}
\includegraphics[width=360pt]{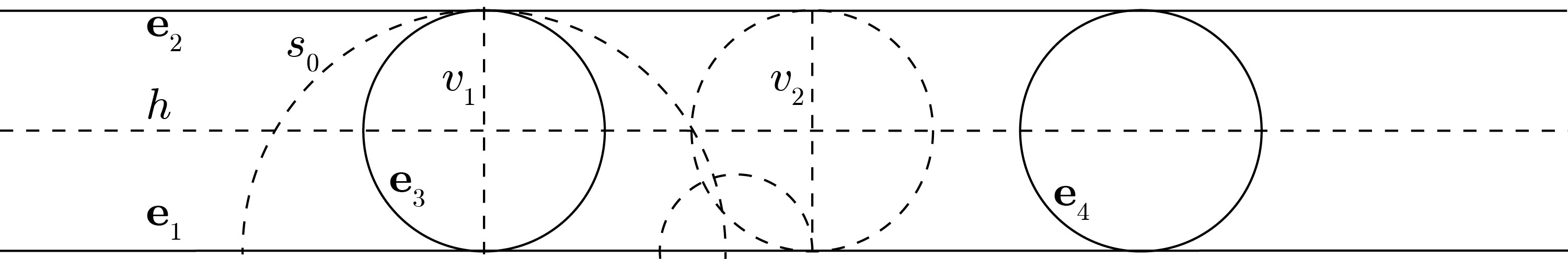}  \\
\includegraphics[width=360pt]{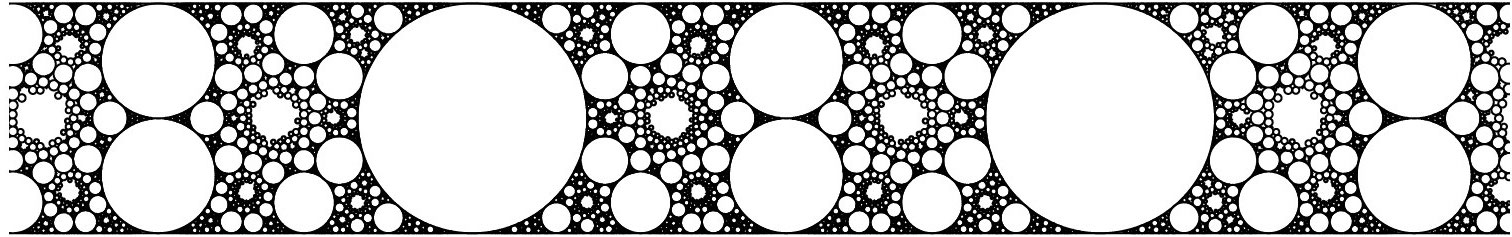}
\caption{\label{fig2} The same configurations as in \fref{fig1}, after inverting in a point of tangency.  The packing is the image of a solid line under the action of the group generated by the inversions and reflections represented by the dotted lines.  In this case $d=1+\sqrt{3}$, so it is not one of the infinite set of packings described in Theorem~\ref{t1}.}
\end{figure}
In the top picture (of \fref{fig2}), we let $d$ be the ratio of the distance between the centers of the two circles and the diameter of the circles.  The game can be won if $d=\sqrt n$ for $n$ a positive integer (see Theorem \ref{t1} below).  Our winning strategy for $n=1$ generates the Apollonian circle packing; for $n= 2$ we get a packing that is described in \cite[(2.3) on p. 394]{Boy74}, and appears in \cite[Figure 3]{Man91} and \cite[Figure 3]{GM10}; and for $n= 3$ we get, modulo inversion, the cross section of the Soddy sphere packing that appears in \cite[Figure 1]{Sod37}.  The case $n=6$ appears in \cite[Figure 7]{Bar18} and \cite[Figure 3]{KN19}.  The cases $n=10$ and $14$ show up in a similar game studied by \cite{CCS19}.  We have not seen the other packings in print.  

While Theorem \ref{t1} tells us that the game can be won for infinitely many $d$, it does not give us a complete strategy.  Solving the problem for successive integers $n=d^2$ is a combination of number theory and geometry, and has a flavor similar to the work of Bianchi \cite{Bia92}.  Once solved, a geometric argument allows us to play a new game:  Suppose we have four mutually tangent circles.  Fill in the remaining space in a nice way.  Contrary to what we had imagined, the Apollonian packing is not the only solution.  See for example the packings in \fref{fig3}.  The first is a blend of the Apollonian packing with Boyd's Example (2.3) \cite{Boy74}, and the second is a blend with a cross section of the Soddy packing.  
\begin{figure}
\begin{center}
\includegraphics{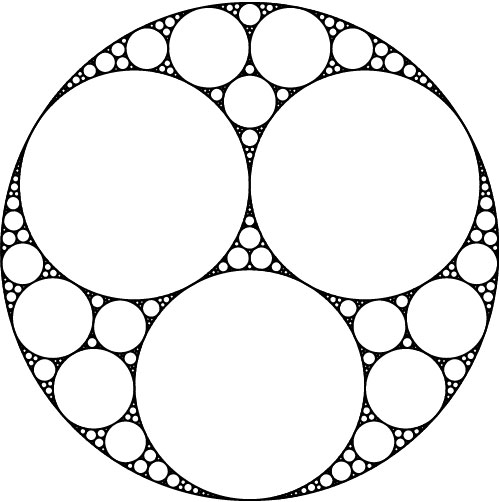} 
\hspace{10pt}
\includegraphics{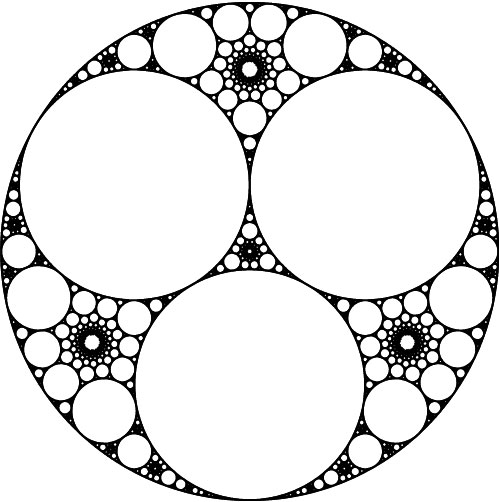}
\caption{\label{fig3} Two circle packings that have the property that every circle is a member of a cluster of four mutually tangent circles.}
\end{center}
\end{figure} 
These packings have what we call the {\em Apollonian property}:  Every circle is a member of a cluster of four mutually tangent circles.  This geometric blending can be done in many ways.     

As in previous works, we think of circles as representing planes in the Poincar\'e upper half-space model of $\Bbb H^3$, a view that dates back to \cite{Max82}.  Packings that are equivalent under inversion are thought of as different perspectives of the same infinite sided ideal polyhedron in $\Bbb H^3$.  Each perspective depends on the choice of point for the point at infinity.  

There are probably many ways of winning this game.  Our approach in Section~\ref{s2} is lattice based and is inspired by results in arithmetic geometry.  For the circle packings guaranteed by Theorem \ref{t1}, there is a perspective so that all circles have integer curvature.  After blending, though, this property is usually lost.       

\subsection*{Acknowledgements} This material is based upon work supported by the National Science Foundation under Grant No. DMS-1439786 and the Alfred P. Sloan Foundation award G-2019-11406 while the first author was in residence at and the second author was visiting the Institute for Computational and Experimental Research in Mathematics in Providence, RI, during the Illustrating Mathematics program.  The first author wishes to thank his home institution, UNLV, for its sabbatical assistance during the Fall of 2019.  The figures in this paper were produced using McMullen's Kleinian groups program \cite{McM}.

\section{Background}
\subsection{The vector model of hyperbolic geometry}
Circles in $\Bbb R^2$ can be represented by $4$-dimensional vectors, an observation that dates back to Clifford and Darboux \cites{Boy73, Boy74}.  The more modern interpretation is that they represent planes in $\Bbb H^{3}$ imbedded in a $4$-dimensional Lorentz space, which in turn represent circles (or lines) on the boundary of the Poincar\'e upper half-space model of $\Bbb H^3$.  

Given a symmetric matrix $J$ with signature $(1,3)$, we define the Lorentz space $\Bbb R^{1,3} $ to be the set of $4$-tuples over $\Bbb R$ equipped with the negative Lorentz product
\[
\vc u\cdot \vc v=\vc u^TJ\vc v.
\]
The surface $\vc x\cdot \vc x=1$ is a hyperboloid of two sheets.  Let us distinguish a vector $D$ with $D\cdot D>0$ and select the sheet $\H$ by:
\[
\H: \qquad \vc x\cdot \vc x=1, \qquad \vc x\cdot D>0.
\]
We define a distance on $\H$ by
\[
\cosh(|AB|)=A\cdot B.
\]
Then $\H$ equipped with this metric is a model of $\Bbb H^{3}$, sometimes known as the {\it vector model}.   Equivalently, one can define
\[
V=\{\vc x\in \Bbb R^{1,3}: \vc x\cdot \vc x>0\}
\]
and $\H=V/\Bbb R^*$, together with the metric defined by
\[
\cosh(|AB|)=\frac{|A\cdot B|}{|A||B|},
\]
where $|\vc x|=\sqrt{\vc x\cdot \vc x}$ for $\vc x\in V$.  For $\vc x\cdot \vc x<0$, we define $|\vc x|=i\sqrt{-\vc x\cdot \vc x}$.

Planes in $\H$ are the intersection of $\H$ with hyperplanes $\vc n\cdot \vc x=0$ in $\Bbb R^{1,3}$.  Such a hyperplane intersects $\H$ if and only if $\vc n\cdot \vc n<0$.  Let $H_{\vc n}$ represent both the hyperplane $\vc n\cdot \vc x=0$ in $\Bbb R^{1,3}$ and its intersection with $\H$.  The direction of $\vc n$ distinguishes a half space
\[
H_{\vc n}^+=\{\vc x: \vc n\cdot \vc x\geq 0\},
\]
in either $\Bbb R^{1,3}$ or $\H$.

The angle $\tt$ between two intersecting planes $H_{\vc n}$ and $H_{\vc m}$ in $\H$ is given by
\begin{equation}\label{eq1}
|\vc n||\vc m|\cos \tt=\vc n\cdot \vc m,
\end{equation}
where $\tt$ is the angle in $H_{\vc n}^+\cap H_{\vc m}^+$.  If $|\vc n\cdot \vc m|=||\vc n||\vc m||$, then the planes are tangent at infinity.  If $|\vc n\cdot \vc m|>||\vc n||\vc m||$, then the planes do not intersect, and the quantity $\psi$ in  $|n||m|\cosh \psi=|\vc n\cdot \vc m |$ is the shortest hyperbolic distance between the two planes.

The group of isometries of $\H$ is given by
\[
\O^+(\Bbb R)=\{T\in M_{4\times 4}: \hbox{$T\vc u\cdot T\vc v=\vc u\cdot \vc v$ for all $\vc u, \vc v\in \Bbb R^{1,3}$, and $T\H=\H$}\}.
\]
Reflection in the plane $H_{\vc n}$ is given by
\begin{equation}\label{eq2}
R_{\vc n}(\vc x)=\vc x-2\proj_{\vc n}(\vc x)=\vc x-2\frac{\vc n\cdot \vc x}{\vc n\cdot \vc n}\vc n.
\end{equation}
The group of isometries is generated by the reflections.

Let $\partial \H$ represent the boundary of $\H$, which is a $3$-sphere.  It is represented by $\L^+/\Bbb R^+$ where
\[
\L^+=\{\vc x\in \Bbb R^{1,3}: \vc x\cdot \vc x=0, \vc x\cdot D>0\}.
\]
Given an $E\in \L^+$, let $\partial \H_E=\partial \H \setminus E\Bbb R^+$.  Then $\partial \H_E$ equipped with the metric $|\cdot |_E$ defined by
\begin{equation}\label{leng}
|AB|_E^2=\frac{\dd^2A\cdot B}{(A\cdot E)(B\cdot E)}
\end{equation}
is the Euclidean plane that is the boundary of the Poincar\'e upper half-space model of $\H$ with $E$ the point at infinity.  The quantity $\dd$ is an arbitrary scaling constant.  In $\partial \H_E$, the plane $H_{\vc n}$ is represented by a circle (or line), which we denote with $H_{\vc n,E}$ (or just $H_{\vc n}$ if $E$ is understood, or sometimes just $\vc n$).

The curvature (the inverse of the radius, together with a sign) of $H_{\vc n, E}$ is given by the formula
\begin{equation}\label{eq3}
\frac{\vc n\cdot E\sqrt 2}{\dd||\vc n||}
\end{equation}
using the metric $|\cdot |_E$ \cite{Bar18}.  Here, $||\vc n||=-i|\vc n|=\sqrt{-\vc n\cdot \vc n}$.    By choosing a suitable orientation for $\vc n$, we get the appropriate sign for the curvature.

\subsection{Circle packings}
The following definitions are due to or inspired by Maxwell \cite{Max82}.  We say $\P \subset \Bbb R^{1,3}$ is a {\em packing} if for all $\vc n,\vc n'\in \P$, there exists a $k>0$ so that $\vc n\cdot \vc n=-k$ and $\vc n\cdot \vc n'\geq k$.  The first condition guarantees that $H_{\vc n}$ is a plane in $\H$, and hence represents a circle on $\partial H_E$.  The second condition guarantees that distinct circles are either tangent or do not intersect.  (Note that if $\vc n\cdot \vc n'< -k$, then the circles do not intersect, but their associated sides overlap.)  We ignore trivial packings of the form $\{\vc n,-\vc n\}$.  

A packing $\P$ is {\em maximal} if we cannot add a vector to $\P$ and have it still be a packing.  Geometrically, this means there is no space left over where one can place another circle (so we sometimes use the term {\em dense}).  

Given a point $E\in \partial \H$, the set 
\[
\P_E=\{H_{\vc n,E}: \vc n\in \P\}
\]
is what is traditionally thought of as a circle packing.  We call $\P_E$ a {\em perspective} of $\P$.  The packings in Figures \ref{fig1} and \ref{fig2} are different perspectives of the same packing.  

A packing is {\em lattice like} if $\P \Bbb Z$ is a lattice in $\Bbb R^{1,3}$.  

\subsection{Lattice based circle packings}
  Suppose 
\[
\LL=\vc e_1\Bbb Z \oplus \vc e_2\Bbb Z \oplus \vc e_3\Bbb Z\oplus \vc e_4\Bbb Z \subset \Bbb R^{1,3}
\]
and 
\[
J_{\LL}=[\vc e_i\cdot \vc e_j]
\]
has integer entries.  Let us fix $k$ (usually $1$ or $2$) so that there exists at least one $\vc n\in \LL$ such that $\vc n\cdot \vc n=-k$; and a $D$ so that $D\cdot D>0$ and $D\cdot \vc n\neq 0$ for all $\vc n\in \LL$ with $\vc n\cdot \vc n=-k$.  Such a $D$ exists.  Let us define
\[
\E_{-k}=\{\vc n\in \LL: \vc n\cdot \vc n=-k, \vc n\cdot D>0\}.
\]
Many of the discs in $\E_{-k}$ overlap, even if they do not intersect.  The set is analogous to the Apollonian super packing \cite{GLM06}.  We can refine this set by defining the cone
\[
\K_{-k}=\bigcap_{\vc n\in \E_{-k}}H_{\vc n}^+
\]
and defining the set
\[
\E_{-k}^*=\{\vc n\in \E_{-k}: H_{\vc n} \hbox{ is a face of $\K_{-k}$}\}.
\]
If for all $\vc n$, $\vc n'\in \E_{-k}^*$, $\vc n\cdot \vc n'\geq k$, then $\E^*_{-k}$ is a packing.  If $E\in \LL$, then, with a suitable choice for $\dd$, every circle in the perspective $\E^*_{-k,E}$ has integer curvature (see Equation (\ref{eq3})). 

Let 
\[
\O_\LL^+=\{T\in \O^+(\Bbb R): T\LL=\LL\}.
\]
Note that $\O_\LL^+$ is an arithmetic group, so it has a convex polyhedral fundamental domain with a finite number of faces and a finite volume.  As a consequence, 
if $\E_{-k}^*$ is a packing, then the packing is maximal.  We sketch the proof:  Let $\vc n$ satisfy $\vc n\cdot \vc n=-k$.  Note that the image of $\vc n$ under the action of $\O_{\LL}^+$ is dense on $\partial \H$, since the fundamental domain for $\O_{\LL}^+$ has finite volume.  If $\E_{-k}^*$ is not maximal, then there exists an $\vc m$ so that $\vc m\cdot \vc m=-k$, $\vc m\cdot D>0$, and $H_{\vc m}^-\subset \K_{-k}$.  By density, there exists an $\vc n\in \E_{-k}$ so that $H_{\vc n}^-\subset H_{\vc m}^-$.  But $\K_{-k}\subset H_{\vc n}^+$, a contradiction.  

\begin{remark}  For a K3 surface $X$ with Picard number $\rho\geq 4$, let $\LL=\Pic(X)$ and let $D$ be ample.  Then $\E_{-2}$ is the set of effective $-2$ divisors on $X$, and $\E_{-2}^*$ is the set of irreducible $-2$ curves on $X$.  The cone $\K_{-2}$ is the ample cone for $X$.    

In general, $\K_{-2}$ does not yield a packing, as it may have edges.  However, using a result of Morrison \cite{Mor84}, there exist plenty of K3 surfaces where this does not happen, including the infinite set described in Theorem \ref{t1}.
\end{remark}

\section{Playing the game}\label{s2}

In \fref{fig2}, we labeled the circles in our initial configuration with $\vc e_1$, ..., $\vc e_4$.  Let us set $\vc e_i\cdot \vc e_i=-2$, so $\vc e_i\cdot \vc e_j=2$ if $i\neq j$ and the circles are tangent (see Equation (\ref{eq1})).  Let $\vc e_3\cdot \vc e_4=a$, which depends on our variable $d$.  Let 
\[
J_{d^2}=[\vc e_i\cdot \vc e_j]=\mymatrix{-2&2&2&2 \\ 2&-2&2&2 \\ 2&2&-2&a \\ 2&2&a&-2}.
\]
The point of tangency between tangent circles $\vc e_i$ and $\vc e_j$ is $\vc e_i+\vc e_j$, so our point at infinity in \fref{fig2} is $E=\vc e_1+\vc e_2$.  Let $P_3$ and $P_4$ be the centers of the circles $\vc e_3$ and $\vc e_4$, so 
\[
 P_i=R_{\vc e_i}(E)=E+4\vc e_i.
 \]
 Using Equations (\ref{eq3}) and (\ref{leng}), $H_{\vc e_3,E}$ has diameter $\dd/2$ and $|P_3P_4|_E=(\dd/4)\sqrt{2+a}$, so 
 \[
 a=4d^2-2.
 \]  
  
\begin{theorem} \label{t1}  Let $n$ be a positive integer and $d=\sqrt n$.  Let $\LL_n$ be the lattice generated by $\{\vc e_1,...,\vc e_4\}$ with $[\vc e_i\cdot \vc e_j]=J_n$.  Then $\E_{-2}^*$ is a maximal circle packing.
\end{theorem}

\begin{proof}  
Let $\vc n, \vc n'\in \E_{-2}^*$.  Then $\vc n\cdot \vc n'$ is an even integer, so we need only show $\vc n\cdot \vc n'\neq 0$.  Suppose $\vc n\cdot \vc n'=0$.  Then
\[
(\vc n-\vc n')\cdot (\vc n-\vc n')=-4.
\]
But $\LL_n$ does not represent $-4$ (that is, there is no $\vc x\in \LL_n$ so that $\vc x\cdot \vc x=-4$), which is easy enough to verify by looking at it modulo $8$.  
\end{proof}

This theorem tells us that the game can be won if $d=\sqrt n$ for $n$  a positive integer, but it does not tell us how to win the game.  The strategy is to find generators for $\O_{\LL_n}^+$.  Some of the reflections shown in \fref{fig2} are in this group for all $n$.  The reflection $R_{h}$ where $h=\vc e_2-\vc e_1=[-1,1,0,0]$ is the symmetry that swaps the first and second components (in this basis).  Similarly, $R_{v_2}$ for $v_2=\vc e_4-\vc e_3=[0,0,-1,1]$ swaps the third and fourth components; and $R_{s_0}$ where $s_0=\vc e_3-\vc e_2=[0,-1,1,0]$ swaps the second and third components.  Thus, they are all in $\O_{\LL_n}^+$.  To solve for $v_1$, we note that $v_1\cdot \vc e_i=0$ for $i=1$, $2$ and $3$, so $v_1=[n,n,1,-1]$.  Though $v_1\cdot v_1=-8n$, the reflection $R_{v_1}$ is nevertheless in $\O_{\LL_n}^+$ for all $n$.  Note that $R_{\vc e_1}$ is also in $\O_{\LL_n}^+$, so the group has a convex fundamental domain $\F_n$ that includes $E$ and is bounded by the faces $H_{\vc e_1}$, $H_{h}$, $H_{v_1}$, $H_{v_2}$, and $H_{s_0}$.  For $n=1$, $2$, and $3$, this region has finite volume, so the group $\langle R_{\vc e_1}, R_{h}, R_{v_1}, R_{v_2}, R_{s_0}\rangle$ has finite index in $\O_{\LL_n}^+$.  The groups are in fact equal (for $n=1,2,3$), but for our purposes, it is enough to find a subgroup of finite index in $\O_{\LL_n}^+$.  This is perhaps easier to demonstrate with an example, but first some preliminaries:  Let us set $\dd=4$, so $H_{\vc e_3,E}$ has unit radius, and the distance $|P_3P_4|_E$ between the centers is $2\sqrt n$.  Let us set $D=\vc e_1+\vc e_2+\vc e_3+\vc e_4$, which is a point a Euclidean distance $\sqrt{n+1}$ above the point $Q_0$ on $\partial H_E$ in the Poincar\'e model (see \fref{fig2.1}), so $\vc n \cdot D\neq 0$ for all $\vc n\in\E_{-2}$ (by Equation (\ref{eq3}) and Theorem \ref{t1}).

\begin{example*}[The case $n=5$]  We first find $\O_{\LL_5}^+$.  We have a head start on the fundamental domain $\F_5$, which is shown in \fref{fig2.1}.  
\begin{figure}
\begin{center}
\includegraphics[width=135pt]{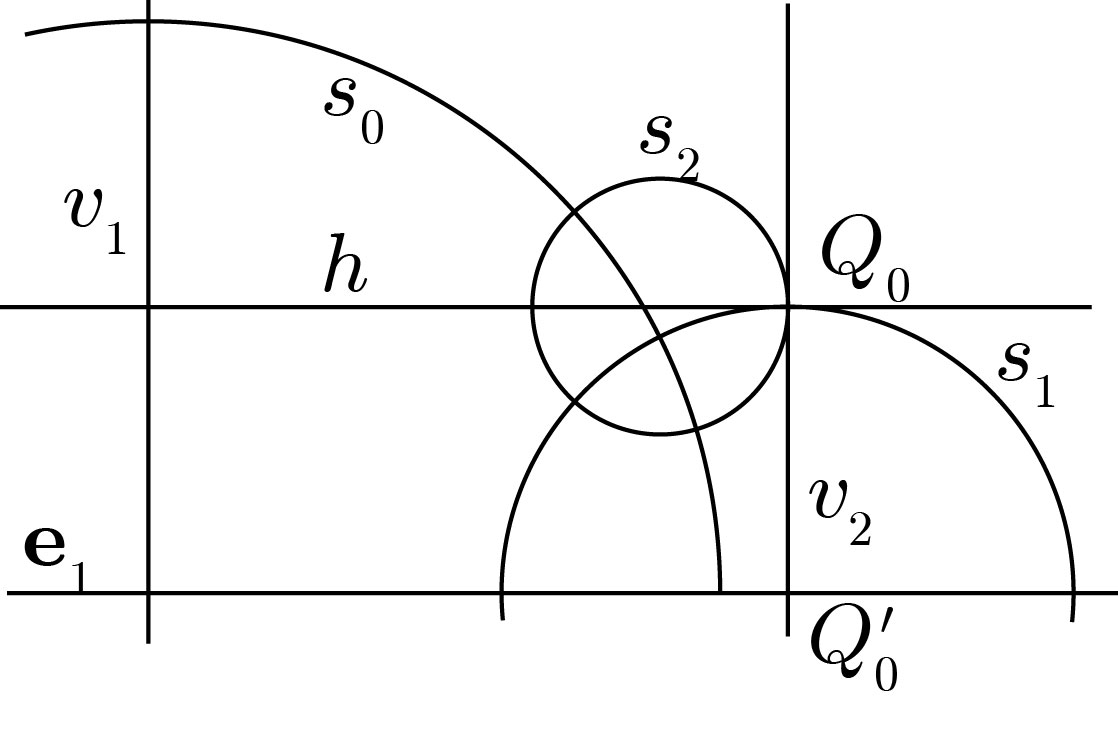}
\caption{\label{fig2.1}  The fundamental domain for $\O_{\LL_5}^+$.  } 
\end{center}
\end{figure}
We need more symmetries of the lattice.  We begin by looking for reflections and taking them to be perpendicular to the faces we already have.  Natural candidates include circles centered at the point labeled $Q_0'$.  The corresponding planes have normal vector a linear combination of $Q_0'$ and $E$.  We find
\[
\vc n_1=[-1,-3,1,1]
\]
gives a symmetry of the lattice.  By design, the circle is centered at $Q_0'$ and from Equation (\ref{eq3}), it has radius $1$.   

It is natural then to assume that $Q_0$ might be a cusp, which would suggest that there is a reflection in a plane that is parallel to $v_2$.  Such planes have normal vectors that are linear combinations of $v_2$ and $Q_0$.  We find
\[
s_2=[-5,-5,3,2]
\]
gives a reflection in $\O_{\LL_5}^+$.  The circle it generates has curvature $\sqrt 5$, is tangent to $v_2$ at $Q_0$ so has center on $h$, and using Equation (\ref{leng}) (applied to its center $R_{s_2}(E)$) we verify that it is the one shown in \fref{fig2.1}.

The resulting region has finite volume, so is a candidate for the fundamental domain of $\O_{\LL_5}^+$.  Our region has two cusps: $E$ and $Q_0$.  We choose our normal vectors to be {\em primitive}, meaning they have coprime coordinates.  The faces at $E$ include one with normal vector $\vc e_1$, which has norm $\vc e_1\cdot \vc e_1=-2$, while the faces at $Q_0$ have normal vectors with norms $-8$, $-8$, $-10$, and $-40$.  Thus, there cannot be a symmetry that sends $Q_0$ to $E$.  If $\O_{\LL_5}^+$ has a symmetry not generated by the faces of our region, then there must exist a symmetry of this region, and since it fixes $E$, it must be a Euclidean symmetry of \fref{fig2.1}.  Clearly no such symmetry exists, so our region is a fundamental domain for $\O_{\LL_5}^+$, and hence
\[
\O_{\LL_5}^+=\langle R_{\vc e_1}, R_h, R_{v_1}, R_{v_2}, R_{s_0}, R_{s_1}, R_{s_2}\rangle.
\]
Thus, $\E_{-2}=\O_{\LL_5}^+(\vc e_1)$, since $\vc e_1$ is the only face with norm $-2$.  (Note that, if a vector $\vc n\in \LL_n$ has  norm $-2$ and intersects a convex fundamental domain for $\O_{\LL_n}^+$, then it must be a face of the fundamental domain, since $R_{\vc n}\in \O_{\LL_n}^+$.)  We set
\[
\CC_5 =\langle R_h,R_{v_1}, R_{v_2}, R_{s_0}, R_{s_1}, R_{s_2} \rangle
\]
(we removed $R_{\vc e_1}$).
Then $\E_{-2}^*=\CC_5(\vc e_1)$.  
\end{example*}

The reflections $R_{s_1}$ and $R_{s_2}$ generalize to all odd $n$.  Let $Q_0'=[4-n,-n,2,2]$ and $s_1=(Q'_0-E)/2$.  Then $R_{s_1}\in \O_{\LL_n}^+$ and on $\partial \H_E$, is inversion in a circle of radius $1$ centered at $Q_0'$.   We use $Q_0=[\frac{1-n}2,\frac{1-n}2,1,1]$ to generate $s_2=nQ_0-v_2$.  Then $R_{s_2}\in \O_{\LL_n}^+$ and on $\partial \H_E$ is inversion in a circle with curvature $\sqrt n$ that is tangent to $v_2$ at $Q_0$.

Something similar works for even $n$ as well.  In this case $Q_0=[1-n,1-n,2,2]$ and $Q_0'=[2-n/2,-n/2,1,1]$.  
If $n\equiv 2 \pmod 4$, then we let $s_1=Q_0'-E$, which has $R_{s_1}\in \O_{\LL_n}^+$ and on $\partial \H_E$ is inversion in a circle of radius $\sqrt 2$ centered at $Q_0'$.  (This finishes the case $n=6$.)  For $n\equiv 0\pmod 4$, we let $s_1=(Q_0-E)/2$.  Then $R_{s_1}\in \O_{\LL_n}^+$, and on $\partial \H_E$ is inversion in a circle of radius one centered at $Q_0$.  (This finishes the case $n=4$.)  Then we can get $s_2=(n/2)Q_0'-v_2$, which generates a reflection in $\O_{\LL_n}^+$ corresponding to a circle of curvature $\sqrt{n/2}$.  However, sometimes one can do better, as in the case $n=12$, where we can let $s_2=(3Q_0-v_2)/2$, which corresponds to a circle with curvature  $\sqrt 3/2$ (and is enough to finish that case).

\begin{example*}[The case $n=7$] We proceed as described above to find $s_1=[-2,-4,1,1]$ and $s_2=[-21,-21,8,6]$ (see \fref{fig5}).  We find a couple more reflections that fill in the remaining gap:  $s_3=[-49,-63,27,15]$ and $s_4=[-4,-4,2,1]$.  
\begin{figure}
\begin{center}
\includegraphics[width=360pt]{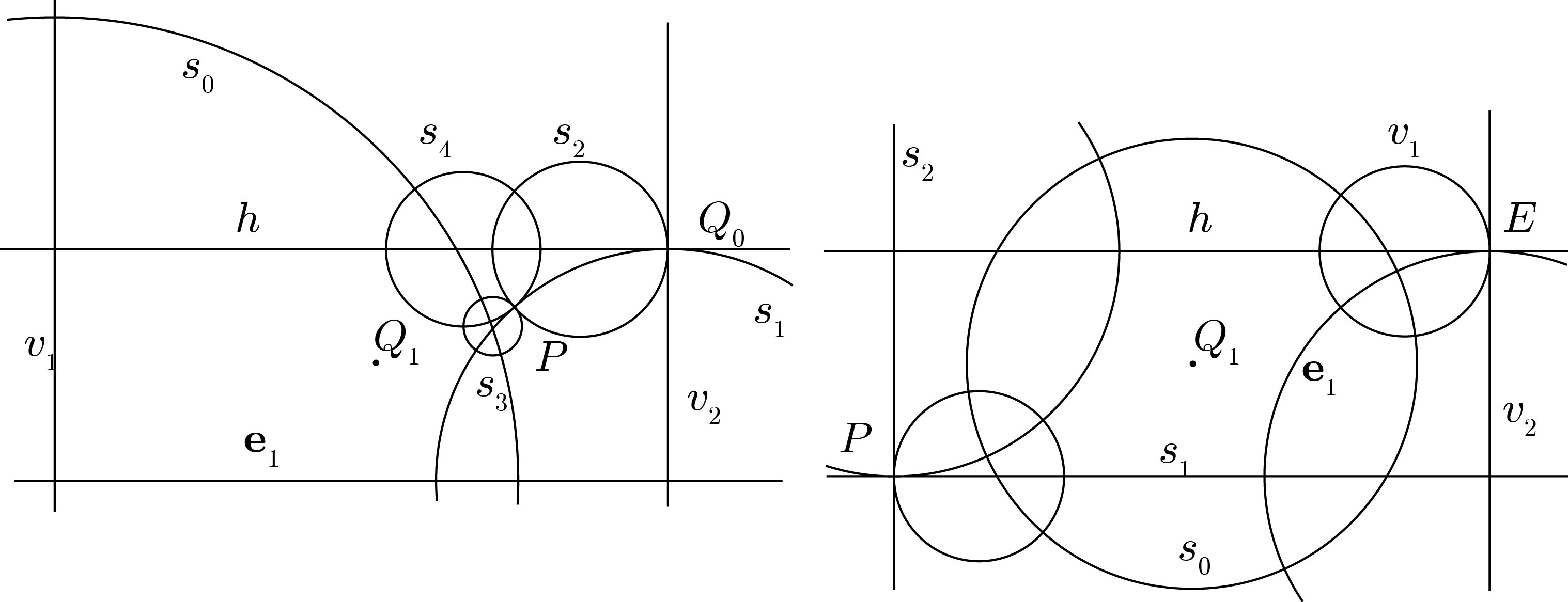}
\end{center}
\caption{\label{fig5} A fundamental domain for a subgroup of $\O_{\LL_7}^+$ and its inversion in  a circle centered at $Q_0$.  }
\end{figure}
Note that $s_4\cdot s_4=-2$, which suggests that the cusps at $P$ and $E$ might be symmetric.  We invert in the point $Q_0$ and get the second region in \fref{fig5}, which appears to have rotational symmetry about the center $Q_1=R_{s_0}(Q_0)$ of $s_0$ (in $\partial \H_{Q_0}$).  A rotation by $\pi$ about a line with endpoints $A$ and $B$ in $\partial \H$ has the equation
\[
\phi_{A,B}(\vc x)=\frac{2((A\cdot \vc x)B+(B\cdot \vc x)A)}{A\cdot B}-\vc x.
\]
The rotation $\phi_{Q_0,Q_1}$ is in $\O^+_{\LL_7}$, and
$\E_{-2}^*=\CC_7(\vc e_1)$ where
\[
\CC_7=\langle R_{h}, R_{v_1}, R_{v_2}, R_{s_0}, \phi_{Q_0,Q_1}\rangle 
\]  
\end{example*}

On $\partial \H_E$, the map $\phi_{A,B}$ is represented by the M\"obius transformation
\begin{equation}\label{mob}
\cc=\mymatrix{ \tilde A+\tilde  B & -2\tilde  A\tilde B \\ 2& -(\tilde  A+\tilde  B)},
\end{equation}
where $\tilde  A$ and $\tilde  B$ are the complex numbers that represent $A$ and $B$ in $\partial \H_E\sim \Bbb C$.    

Another useful formula gives the distance $x$ from a point $A$ to a line $H_{\vc n}$ (so $\vc n\cdot E=0$):
\[
x=\frac{\dd A\cdot \vc n}{\sqrt{-2\vc n\cdot \vc n}A\cdot E}.
\]
So, for example, if we think of $\vc e_1$ as the $x$-axis and $v_1$ as the $y$-axis, then the coordinates $(A_x,A_y)$ of a point $A\in \partial \H_E$ is
\[
(A_x,A_y)=\left( \frac{A\cdot v_1}{\sqrt n A\cdot E},\frac{2 A\cdot \vc e_1}{A\cdot E}\right).
\]
A plane with normal vector $\vc n$ is represented by a circle with radius given by Equation (\ref{eq3}) and center
\[
\left(\frac{\vc n\cdot v_1}{\sqrt n \vc n\cdot E}, \frac{2 \vc n\cdot \vc e_1}{\vc n\cdot E}\right).
\]

The point $Q_1$ generalizes for $n\equiv 3 \pmod 4$:
\begin{equation}\label{Q1}
Q_1=[ {-n^2+25},{-n^2+9},{2n+10},{2n-6}]
\end{equation}
(or an equivalent scalar multiple that is primitive).  Then $\phi_{Q_0,Q_1}\in \O_{\LL_n}^+$.  Geometrically, $Q_1$ is the center of the rectangle formed by $h$, $v_2$, $s_1$ and $s_2$ when $Q_0$ is the point at infinity.  

\begin{example*}[The case $n=10$]  The group $\O_{\LL_{10}}$ is generated by the usual reflections, the reflection $n_1=[-4,-6,1,1]$, and the map
\[
-R_{P}(\vc x)=2\frac{P\cdot{\vc x}}{P\cdot P}P-\vc x
\]
for $P=[6,6,-3,-1]$.  Note that $P\cdot P=16$, so $P$ is a point in the hyperbolic space.  This map is the $-1$ map on $\H$ centered at $P$.  As an action on $\partial \H_E$, it is the composition of inversion in a circle and rotation by $\pi$ about its center.  The center is $-R_P(E)$, and Equation (\ref{eq3}) gives $i$ times the circle's curvature\footnote{McMullen's code does not allow for a map like this.  This packing can be generated by a subgroup of index two that is generated by just reflections, though we also edited his code.} (see \fref{fig6}).
\begin{figure}
\begin{center}
\includegraphics{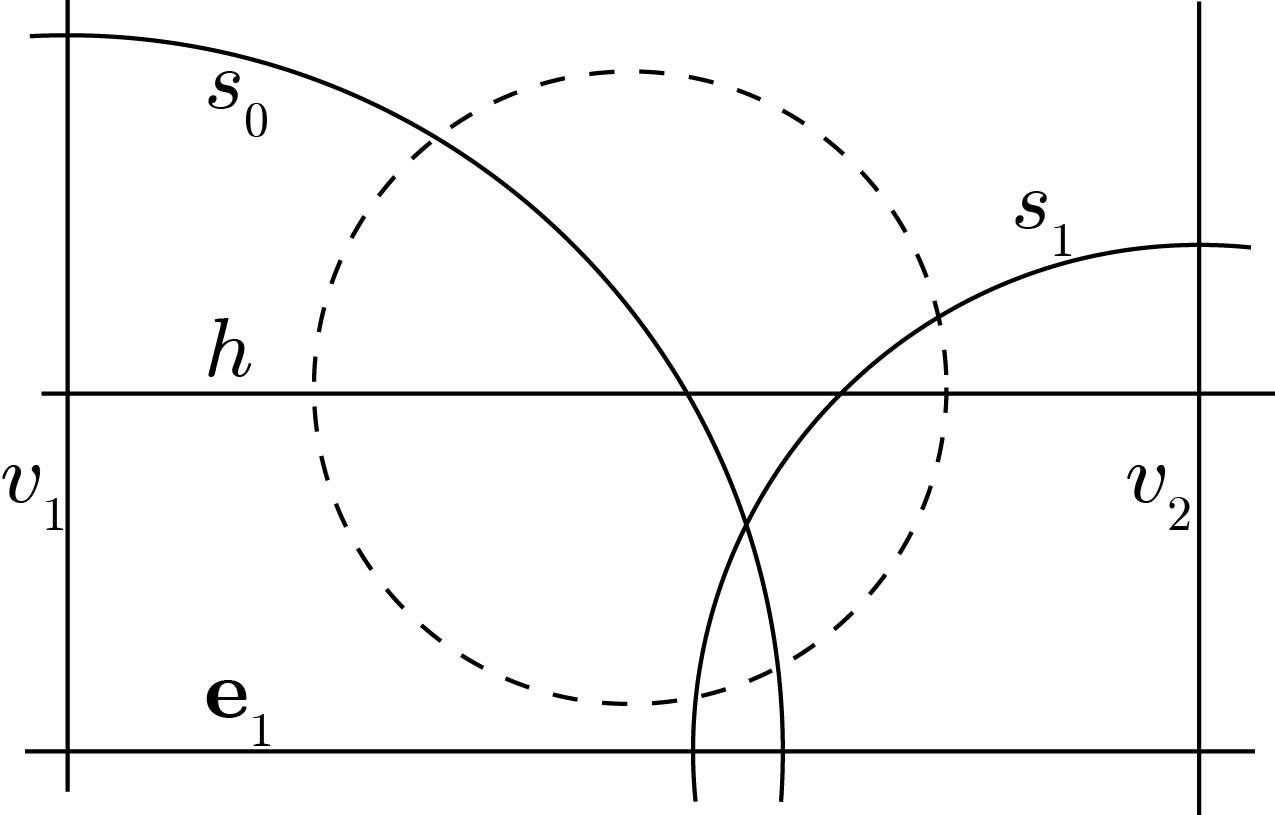}
\end{center}
\caption{\label{fig6} The fundamental domain for $\O_{\LL_{10}}$.  The dotted circle represents a map that is inversion in that circle composed with rotation by $\pi$ about its center.}
\end{figure}   
\end{example*}

\begin{example*}[The case $n=11$]  Every good game has its boss levels, and $11$ is one of them.  Chasing reflections is a never ending pursuit that leads one into the cusp $Q_1=[-6,-7,2,1]$ (see \fref{fig7} and also Equation (\ref{Q1})).   
\begin{figure}
\begin{center}
\includegraphics[width=\textwidth]{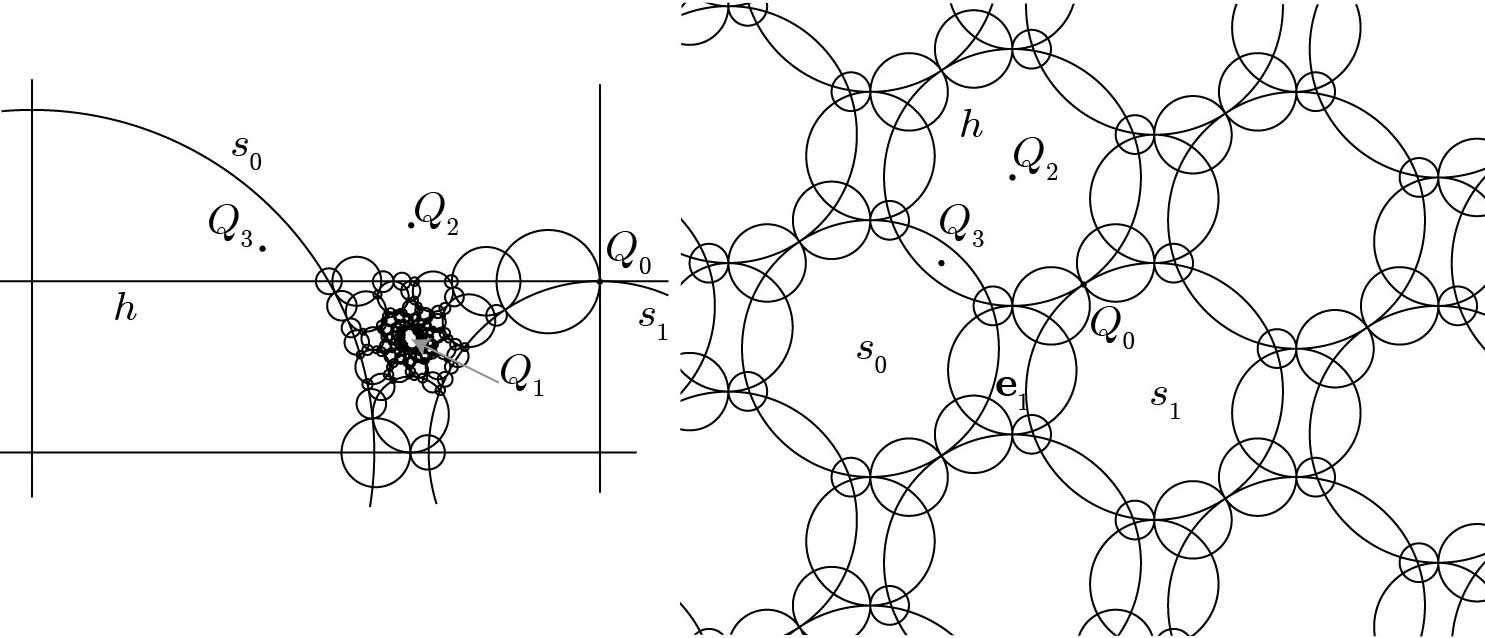}
\end{center}
\caption{\label{fig7}  A never ending set of reflections that generate the tiling for $n=11$, converging on the cusp $Q_1=[-6,-7,1,1]$ (left).  On the right, the same image inverted in $Q_1$, illuminating the rotational symmetry about $Q_0$, $Q_2$ and $Q_3$.}
\end{figure}
Inverting in this point reveals the rotational symmetry about $Q_0$; $Q_2=R_h(Q_1)=[-7,-6,2,1]$; and $Q_3=[-8,-7,4,1]$, which is a linear combination of $h+s_0$ and $Q_1$.  These three rotations, together with the usual reflections, generate $\O_{\LL_{11}}^+$.  
\end{example*}

\begin{example*}[The case $n=15$]  There is a point $P=[-6,-9,2,1]\in \H$ such that $-R_P\in \O_{\LL_{15}}^+$, but since $P\cdot \vc e_1=0$, it is not a symmetry of $\E_{-2}^*$.  We compose with $R_{\vc e_1}$ to get a rotation by $\pi$ whose endpoints are the irrational points $Q_2$ and $Q_3$ in Table \ref{tab1}.  Together with $\phi_{Q_0,Q_1}$ (see Equation (\ref{Q1})) and the usual reflections, these generate the packing.
\end{example*}

\begin{example*}[The case $n=21$]  This is the first group that includes a glide translation:
\[
T=\mymatrix{-24 & -68 & -69 & -1182 \\ -24 & -69& -68 & -1202 \\ 3 & 8 & 8& 142 \\ 2&6&6&103}.
\]
This was found by finding a cusp similar to $Q_0$ and guessing that there should be a symmetry that sends one to the other.  The group is generated by the usual reflections including $R_{s_1}$ and $R_{s_2}$, the map $T$, and the reflection $R_{s_3}$ where $s_3=[-12,-15,2,1]$.  The eigenvalues of $T$ are $1$, $-1$, $\ll=9+4\sqrt 5$, and $\ll^{-1}$.  
\begin{figure}
\includegraphics[width=162pt]{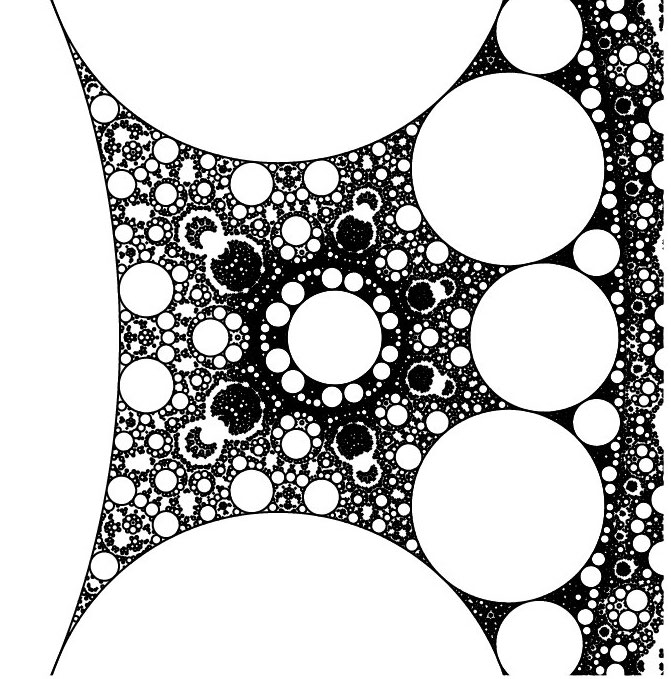}
\hspace{10pt} 
\includegraphics[width=162pt]{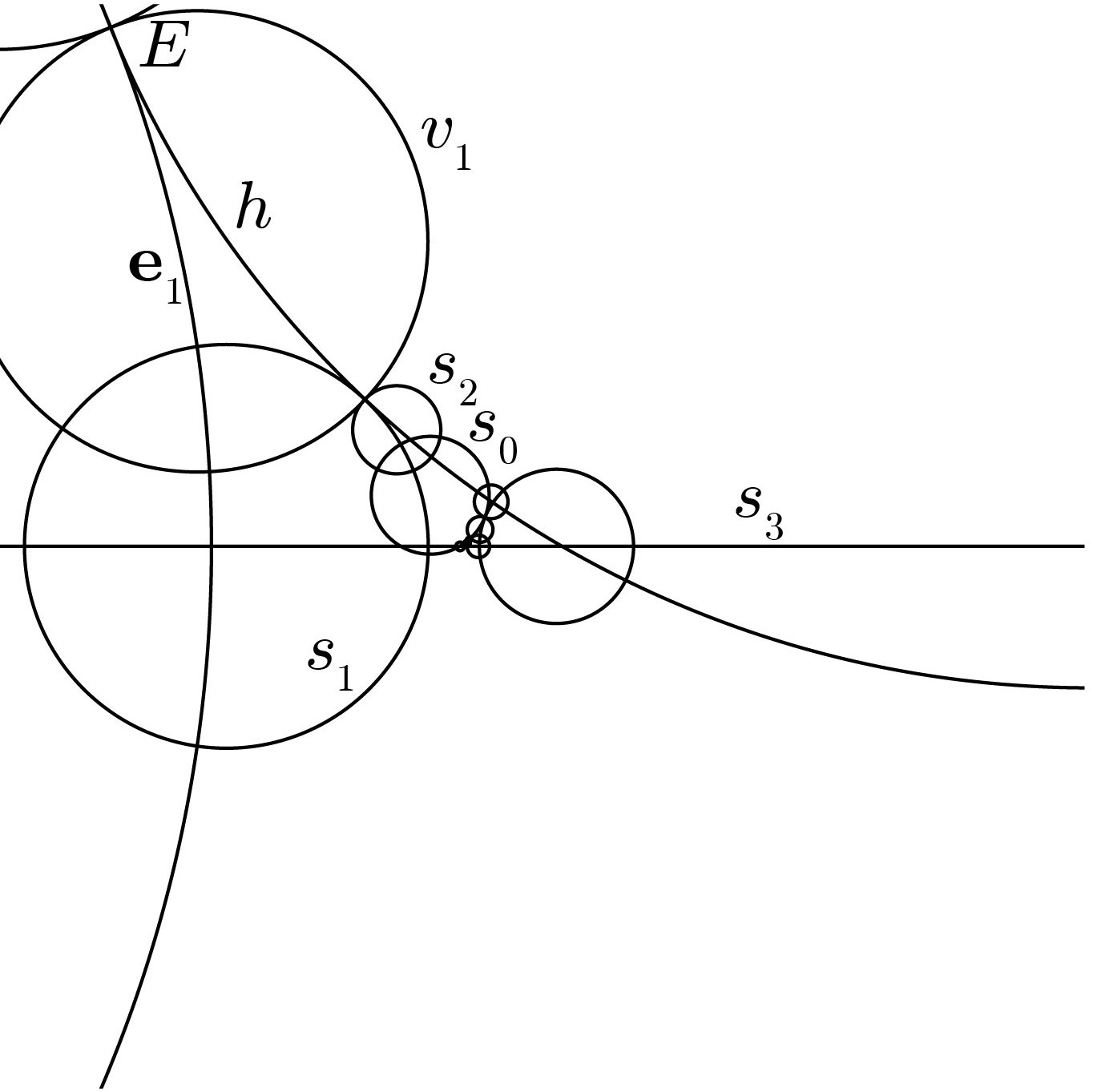}
\caption{\label{fig21}  The tiling with $n=21$ and the eigenvector $A$ at infinity. The map $T$ reflects the central circle along a vertical axis and dilates it by a factor of $\ll$, giving the large circle $\vc e_1$ bounding the picture on the left. The figure on the right represents various reflective symmetries of the packing.}
\end{figure}
Let $A$ be the eigenvector associated to $\ll$.  A perspective with $A$ the point at infinity is shown in \fref{fig21}.   
\end{example*}

A hyperbolic translation has an eigenvalue $\ll>1$, and the rest are $\ll^{-1}$ and $1$ with multiplicity $2$.  Let $A$ and $B$ be the eigenvectors associated to $\ll$ and $\ll^{-1}$, respectively.  Then the corresponding map on $\Bbb C$ is
\[
\tau_{\ll,A,B}(z) =\frac{(\ll \tilde A-\tilde B)z+  (1-\ll) \tilde A\tilde B}{ (\ll-1)z + \tilde A-\ll \tilde B}.
\]
For the above case ($n=21$), rather than come up with a representation for the glide reflection $T$, let us note that the eigenvector associated to $1$ is $s_3$, so the composition $S=R_{s_3}\circ T$ is an orientation preserving map that is in our group.  It is the composition of a hyperbolic translation with rotation by $\pi$ about its line of translation, and is therefore represented by $\ss(z)=-\tau_{-\ll, A,B}(z)$.

\ignore{
Chasing reflections in this case is a never ending pursuit that leads one to the cusp $Q_1=[-6,-7,2,1]$ (see \fref{fig7}).  
\begin{figure}
\begin{center}
\includegraphics[width=\textwidth]{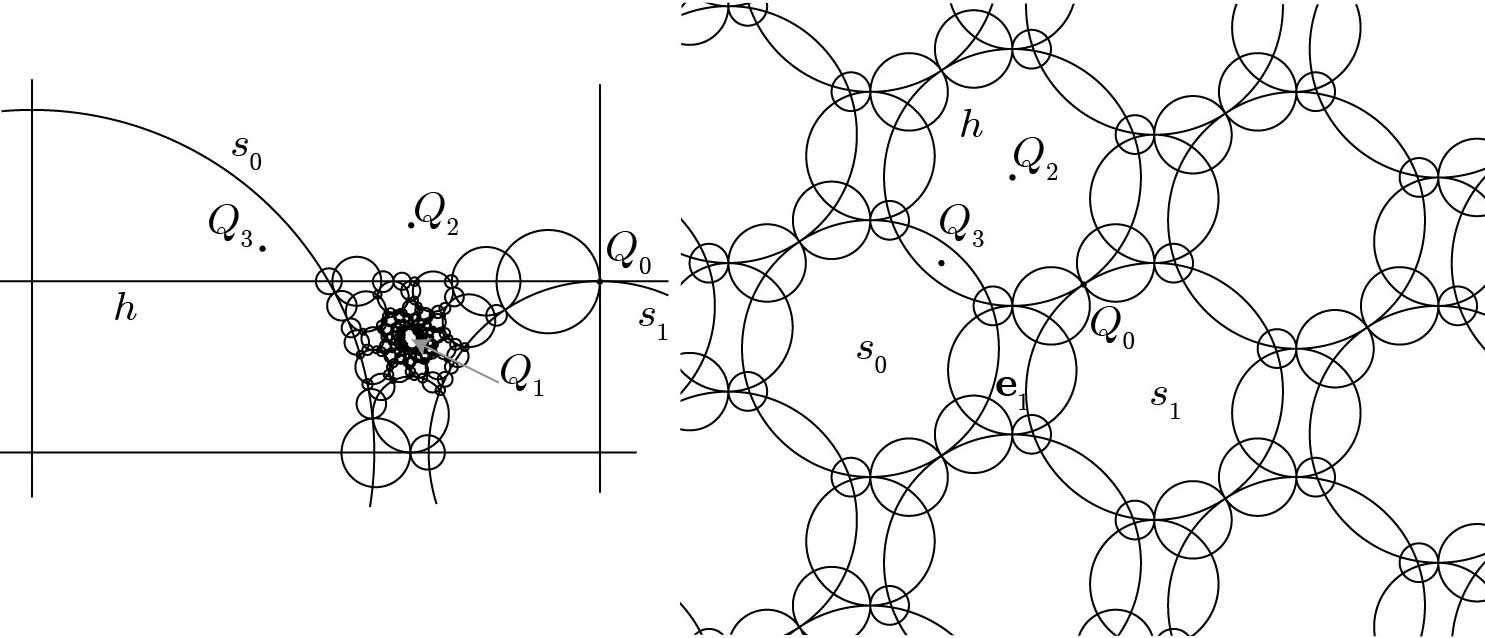}
\end{center}
\caption{\label{fig7}  A never ending set of reflections that generate the tiling for $n=11$, converging on the cusp $Q_1=[-6,-7,1,1]$ (left).  On the right, the same image inverted in $Q_1$, illuminating the rotational symmetry about $Q_0$, $Q_2$ and $Q_3$.}
\end{figure}
Moving $Q_1$ to infinity leads us to the rotations $\phi_{Q_1,Q_0}$, $\phi_{Q_1,Q_2}$ and $\phi_{Q_1,Q_3}$ for $Q_2=R_h(Q_1)=[-7,-6,2,1]$, and $Q_3$ a linear combination of $h+s_0$ and $Q_1$: $Q_3=[-8,-7,4,1]$.  These three rotations, together with the usual reflections, generate $\O_{\LL_{11}}^+$.  
}
\ignore{
(For the interested reader, note that there is only one cusp when $n=1$ or $2$, and that the two cusps for $n=3$ are not congruent, since the norms of their faces are different.  Thus, if there is another symmetry, then it is a Euclidean symmetry of the rectangular base.  These are easily eliminated.)  

Once we find generators for $\O^+_{\LL_n}$ and its corresponding fundamental domain $\F_n$, we look at the subgroup $\CC_n$ generated by the elements

******

For $n\geq 4$, we will have to find more generators.  In the following, we look at certain interesting cases.  Results for all $n\leq 26$ are given in the Appendix.

***

\begin{remark}  The distance $d$ for the packing in \fref{fig2} is $1+\sqrt 3$, so it is not a member of this infinite class. 
\end{remark} 
}

\section{Gluing and slicing}\label{s3}  
\subsection{Let us play a new game.}  Find a maximal circle packing with the property that every circle in the packing is a member of a cluster of four mutually tangent circles.  Of course, the Apollonian circle packing satisfies this property, so the challenge is to be different.  

Let us begin, for example, with the fundamental domain for $n=7$ shown in \fref{fig6}, and let us move $v_1$ to the left one unit (let us call that new line $v_1'$).  This gives us a new group $\CC_7'=\langle R_{v_1'},R_h,R_{v_2}, R_{s_0},\phi_{Q_0,Q_1}\rangle$ and $\CC_7'(\vc e_1)$ is the packing shown in \fref{fig7A}.
\begin{figure}
\includegraphics[width=\textwidth]{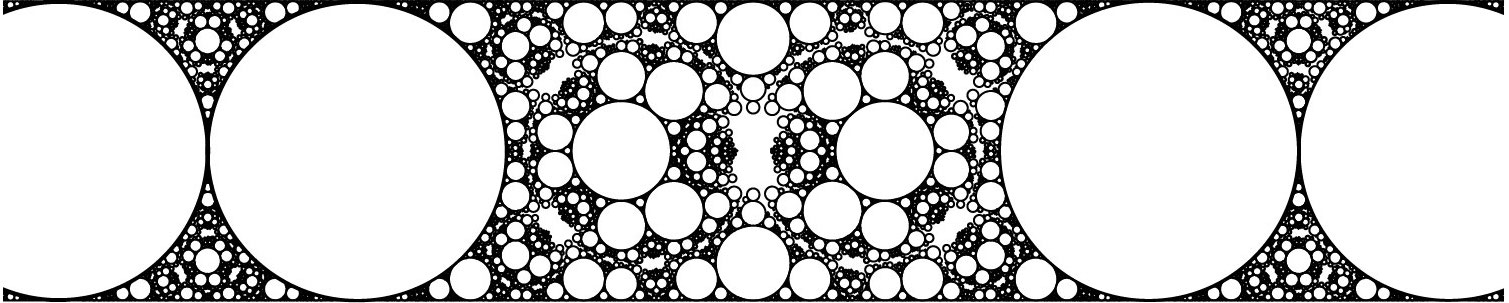}
\caption{\label{fig7A}  A blend of the $n=7$ packing and the Apollonian packing}
\end{figure}
This packing clearly has a cluster of four mutually tangent circles, and since $\CC_7'$ acts transitively on the packing (it is the orbit of a single element), every circle is a member of a cluster of four mutually tangent circles.  We say that the packing has the {\em Apollonian property}.  

We should think of this as gluing two compatible fundamental domains together.  (A similar process in described in \cite{CCS19}.)  To the left of the fundamental domain for $n=7$, we glued along the plane $H_{v_1}$ a reflected version of our fundamental domain for the Apollonian packing.  Since the faces that intersect $H_{v_1}$ are compatible, the new fundamental domain generates a circle packing.  We can  do this for any $n$, giving us the following result:

\begin{theorem}  There exists an infinite number of maximal circle packings with the Apollonian property.
\end{theorem}

\begin{proof}  There is one minor detail we should address:  For a given $n$, we do not know whether $\CC_n$ acts transitively on $\E_{-2}^*$, nor whether this is the case for infinitely many $n$.  However, suppose the fundamental domain $\F_n$ has a face other than $\vc e_1$ whose norm is also $-2$.  Then the reflection through that face is in $\O_{\LL_n}^+$ and hence by adding that reflection to $\CC_n$, we can get a different subset of $\E_{-2}$ that is a maximal circle packing.  By doing this for all faces of $\F_n$ except for $\vc e_1$, we get a group that acts transitively on the new packing.  Thus, when we replace $R_{v_1}$ with $R_{v_1'}$ for this modified packing, we get a packing with the Apollonian property.  
\end{proof}

We can glue on the face $v_2$ as well, when the two fundamental domains are compatible.  For example, in \fref{fig57}, we have glued the fundamental domains for $n=5$ and $n=7$ together in two different ways, giving us two different packings.  
\begin{figure}
\begin{center}
\includegraphics[width=301pt]{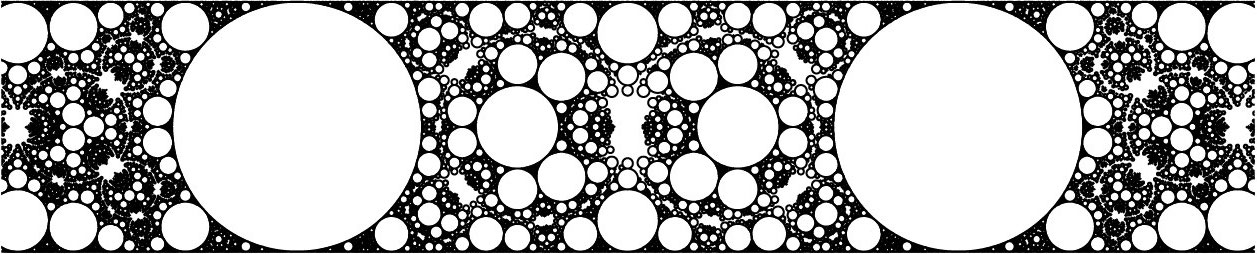}
\includegraphics[width=301pt]{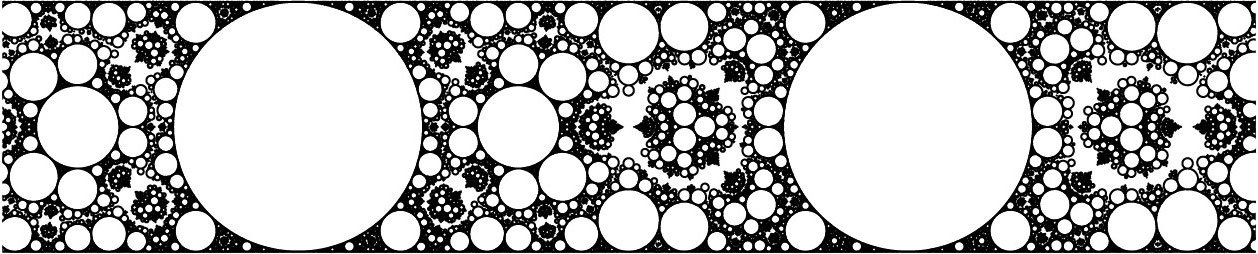}
\end{center}
\caption{\label{fig57}  A blend of the $n=5$ and $n=7$ packings, glued on the face $v_1$, and on the face $v_2$.}
\end{figure}

\begin{remark}  Gluing is a geometric process, so it is no surprise that the integral curvature property is lost when two fundamental domains are glued together.  But not always, so let us suggest a new rule/game:  Find an infinite set of maximal circle packings that have both the Apollonian property and the integral curvature property.
\end{remark}   
 
\subsection{Closing the gap?}  Is there a strategy to create infinitely many packings where the gap $d$ is between $1$ and $\sqrt 2$?  For example, if we let $n=3/2$, then $J_n$ still has integer entries so we can investigate $\E_{-2}^*$.  
\begin{figure}
\begin{center}
\includegraphics[width=108pt]{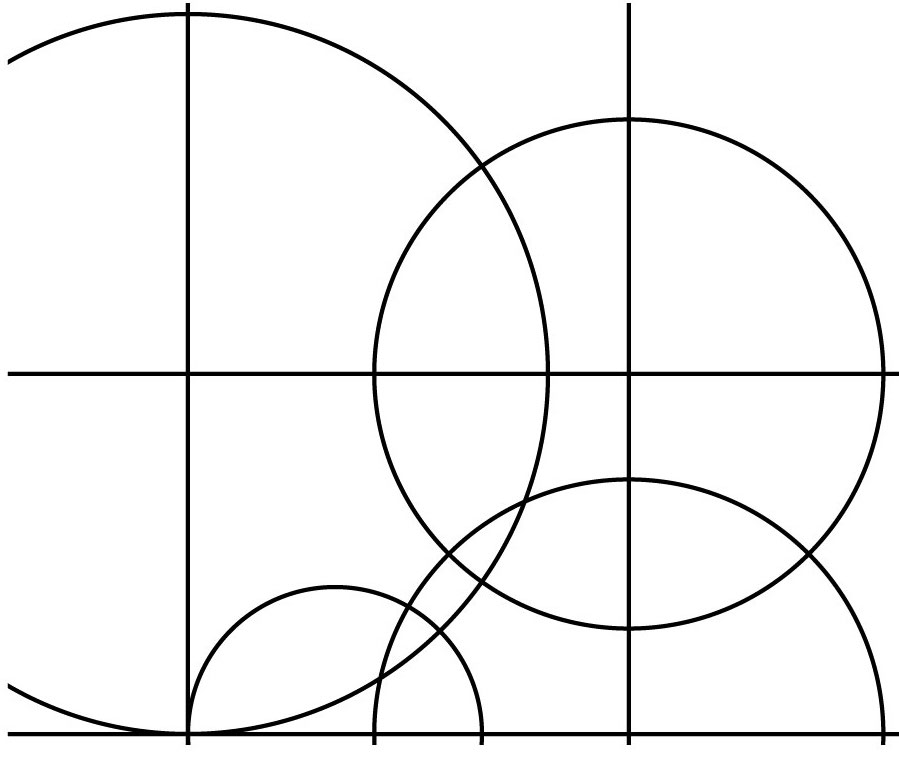} \ \ 
\includegraphics[width=108pt]{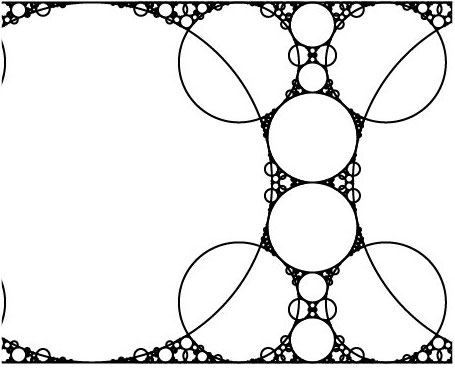} \ \ 
\includegraphics[width=108pt]{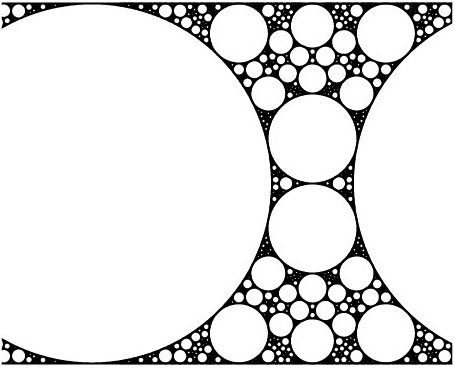} 
\end{center}
\caption{\label{fig32} The fundamental domain for $\O_{\LL_{3/2}}^+$, the set $\E_{-2}^*$, and a packing that is a subset of $\E_{-2}$.}
\end{figure}
A fundamental domain for $\O_{\LL_{3/2}}^+$ is shown in \fref{fig32}, and because it includes a face that is not perpendicular to $H_{\vc e_3}$, the set $\E_{-2}^*$ is not a packing.  However, if we reflect our fundamental domain across that face and glue the two domains together, we get a subgroup of index two in the symmetries of $\E_{-2}^*$ that generates a packing.  

We have not played this game long enough to know if there is a strategy that gives us packings where the gap $d$ converges to $1$.  

\subsection{Filling in ghost circles.}  When a packing has a ghost circle (e.g. $n=7$ in \fref{fig1to9}), we can fill it in or reflect across it (see \fref{fig7a}).  This corresponds to slicing off a portion of the fundamental domain and including the resulting new face as a member of the new packing, or reflecting across it.  In \fref{fig5} (right), this corresponds to cutting the region in half with a line through $Q_1$ perpendicular to $h$.      
\begin{figure}
\begin{center}
\includegraphics[width=301pt]{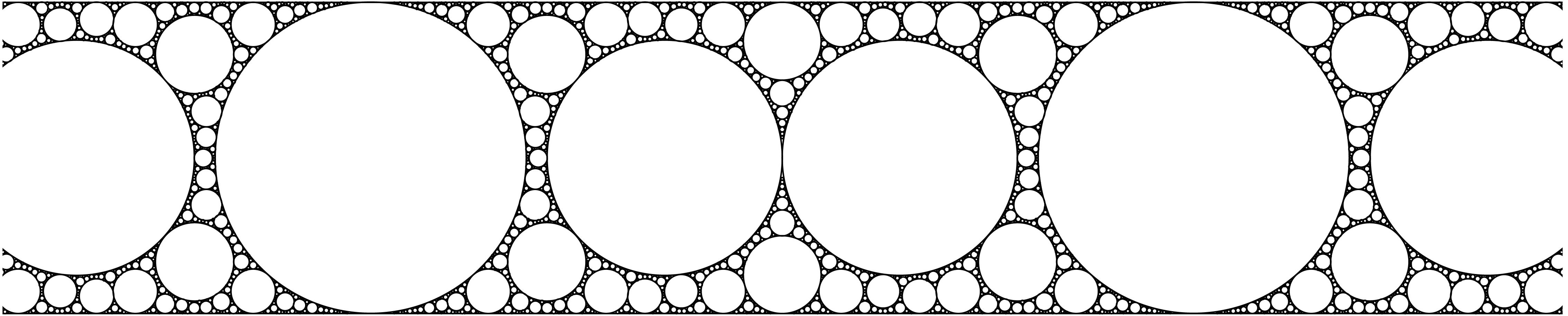}
\includegraphics[width=301pt]{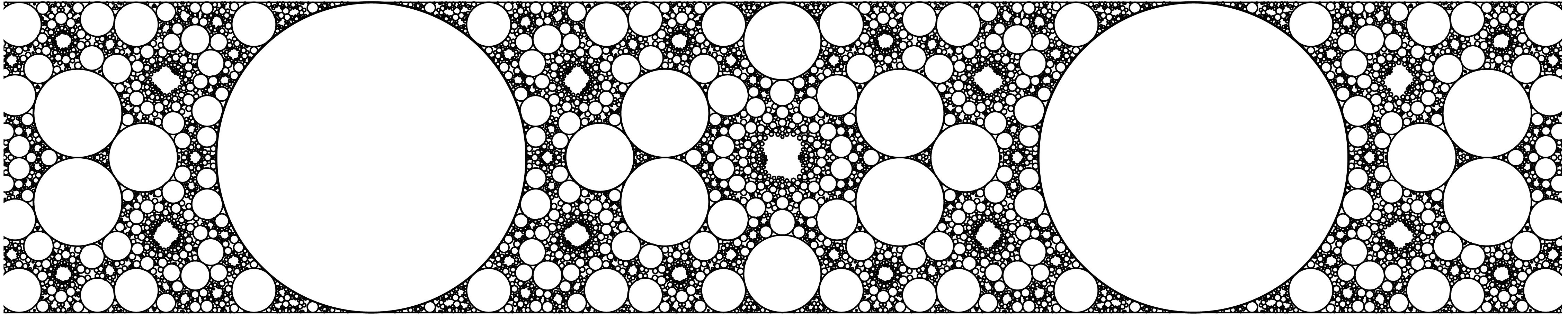}
\end{center}
\caption{\label{fig7a} Two variations on the $n=7$ packing.}
\end{figure}
\ignore{
The symmetries of $\O_{\LL_{3/2}}^+$ include the usual $R_{\vc e_1}$, $R_h$, $R_{v_1}$ and $R_{v_2}$ (but not $R_{s_0}$), and also $R_{\vc e_3}$, and $R_{s_i}$ for $i=1, 2, 3$ where 
\begin{align*}
s_1&=[] \\
s_2&=[]
s_3&=[].
\end{align*}  
Then $ 
}

\section{Appendix}

\begin{figure}
\includegraphics[width=301pt]{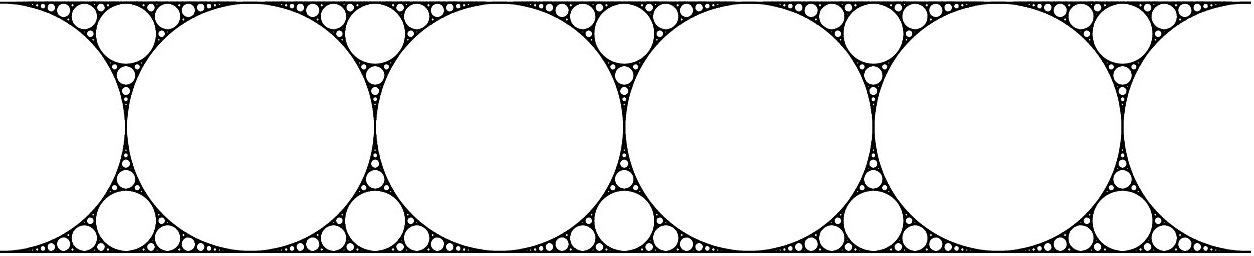}
\includegraphics[width=301pt]{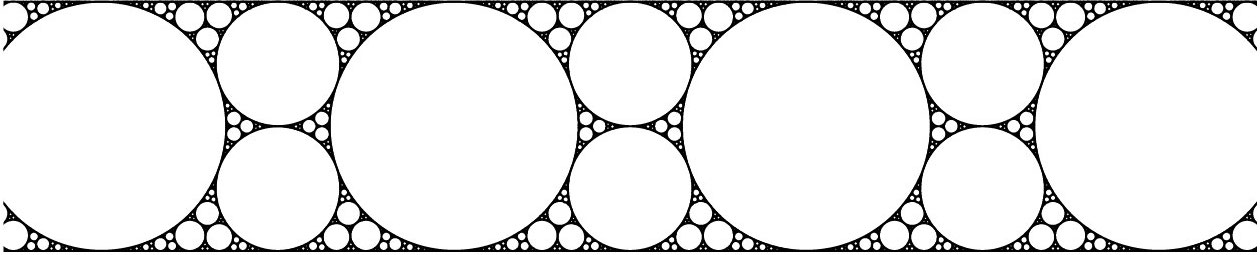}
\includegraphics[width=301pt]{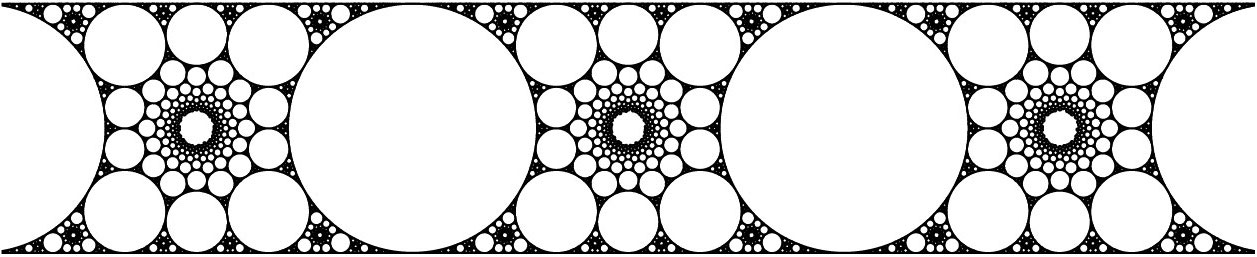}
\includegraphics[width=301pt]{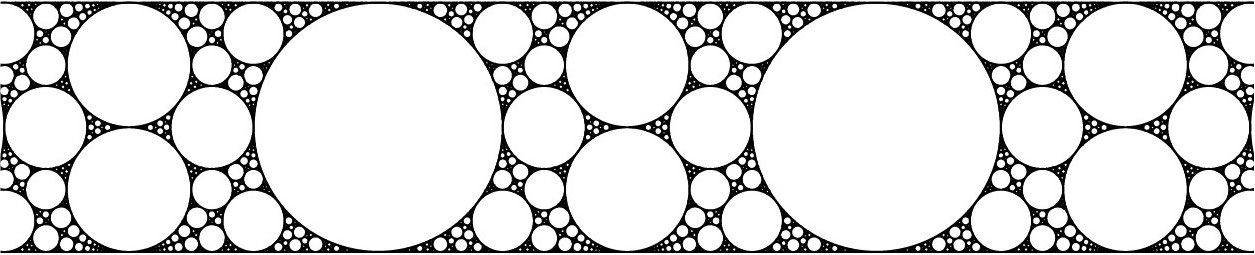}
\includegraphics[width=301pt]{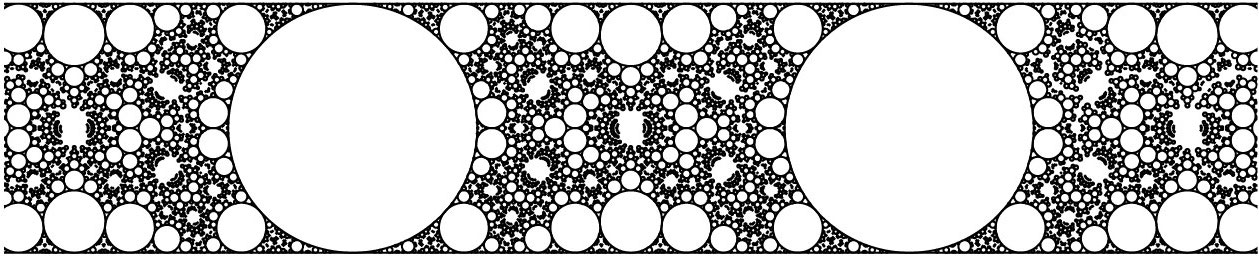}
\includegraphics[width=301pt]{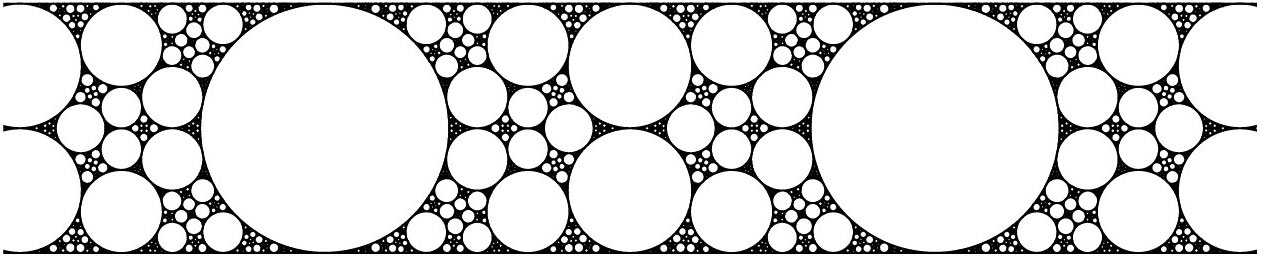}
\includegraphics[width=301pt]{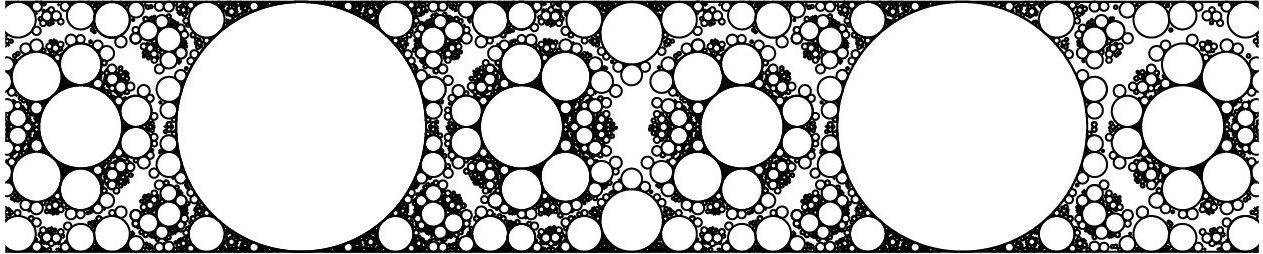}
\includegraphics[width=301pt]{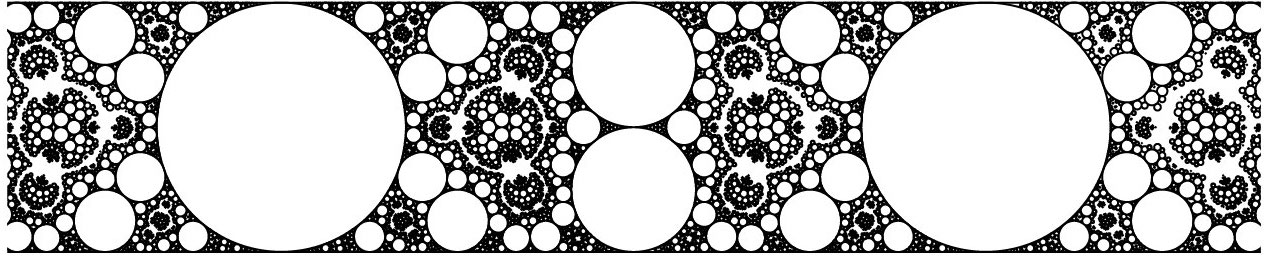}
\includegraphics[width=301pt]{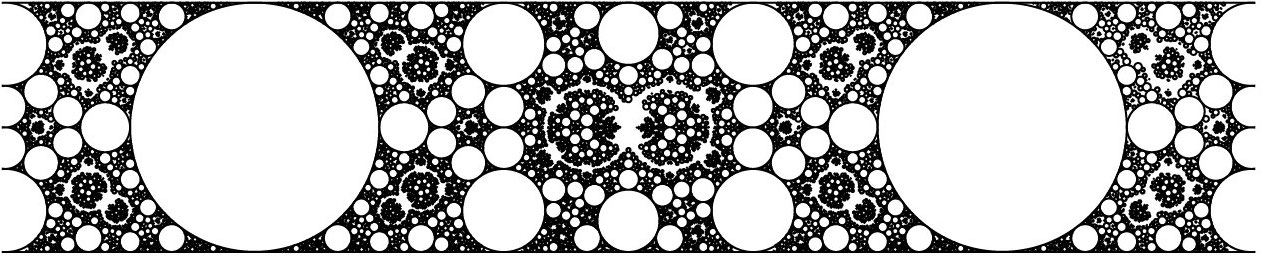}
\caption{\label{fig1to9} The packings for $n=1$ through $9$.  }
\end{figure}

\begin{figure}
\includegraphics[width=301pt]{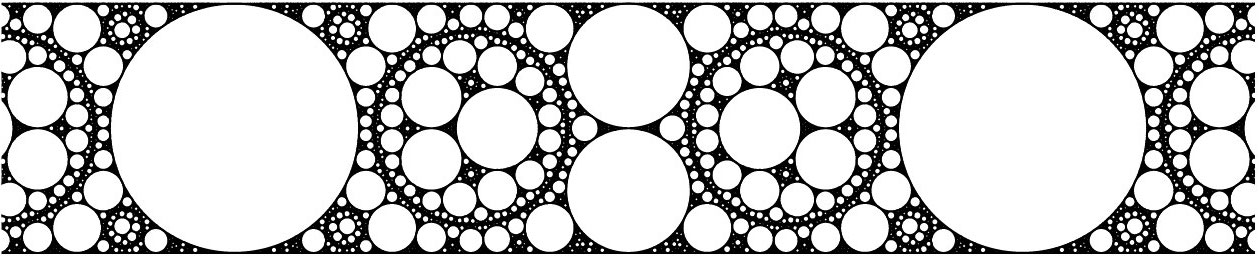}
\includegraphics[width=301pt]{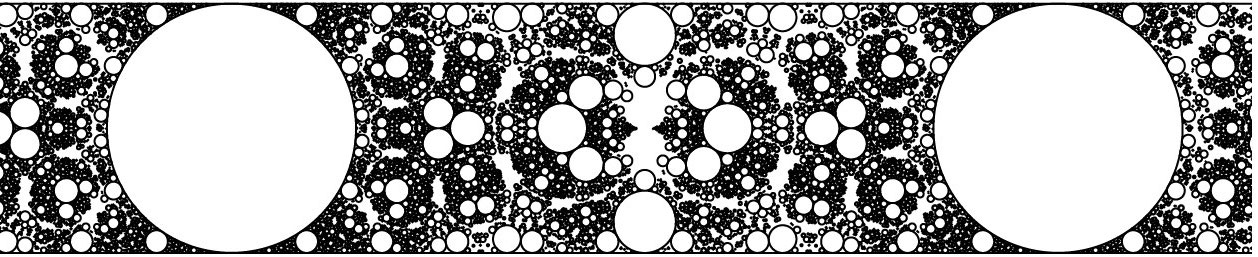}
\includegraphics[width=301pt]{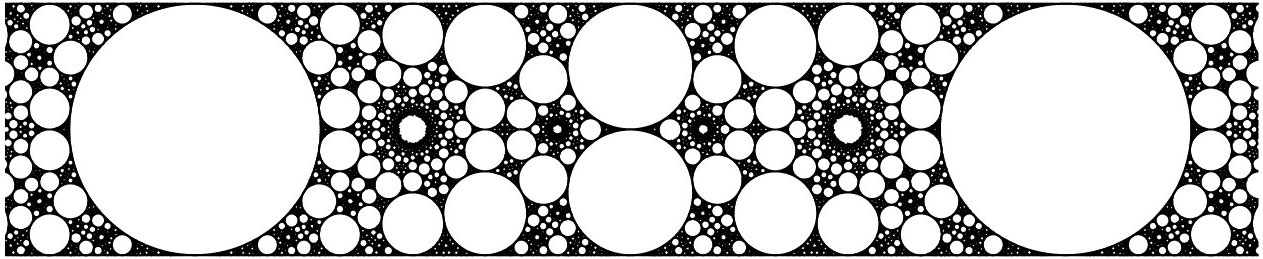}
\includegraphics[width=301pt]{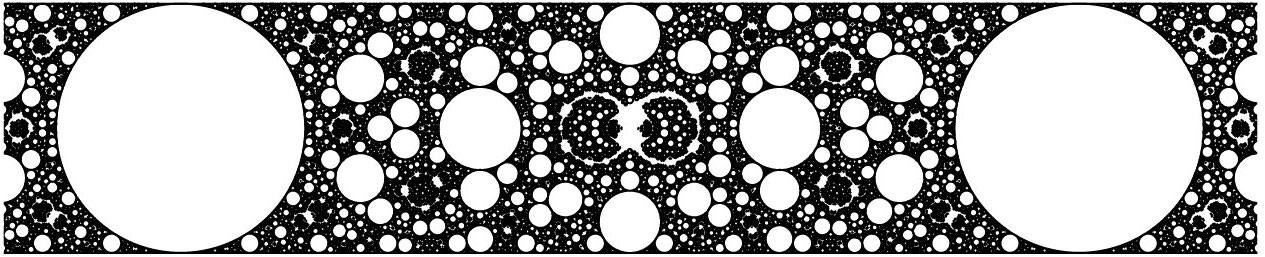}
\includegraphics[width=301pt]{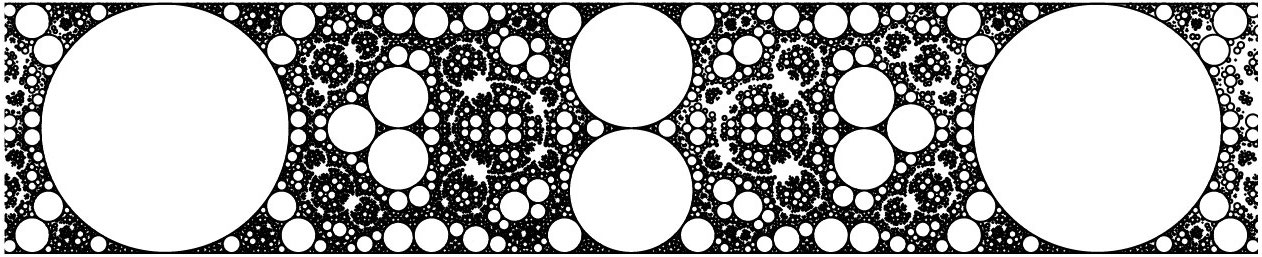}
\includegraphics[width=301pt]{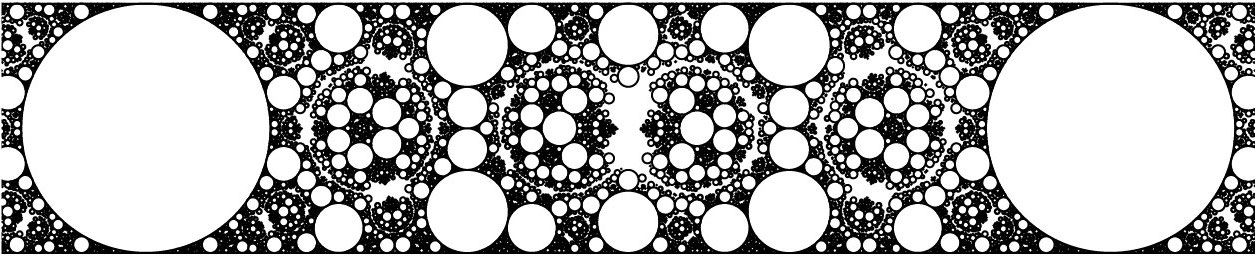}
\includegraphics[width=301pt]{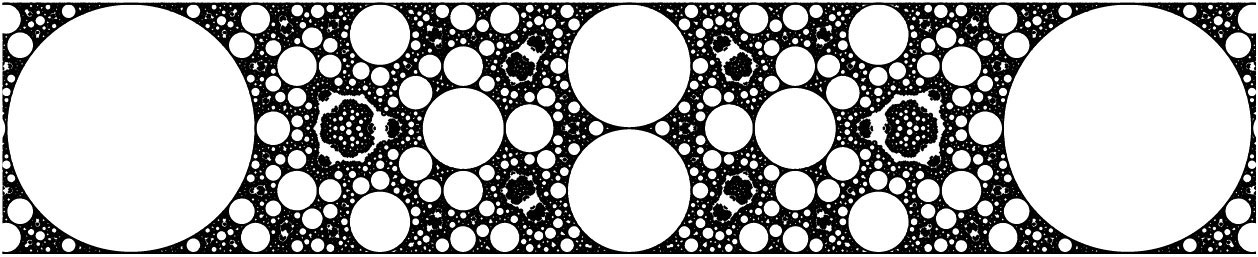}
\includegraphics[width=301pt]{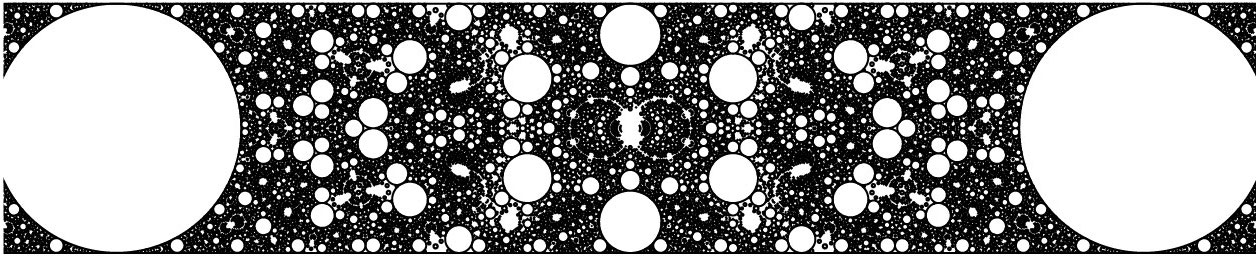}
\includegraphics[width=301pt]{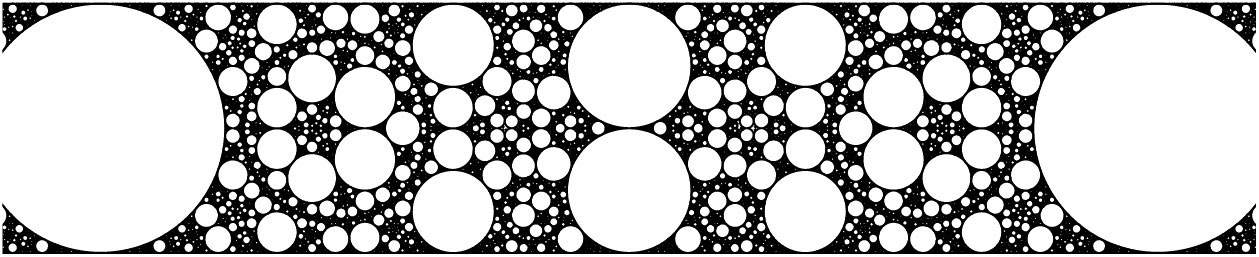}
\caption{\label{fig10to18}  The packings for $n=10$ through $18$.  }
\end{figure}

\begin{figure}
\includegraphics[width=301pt]{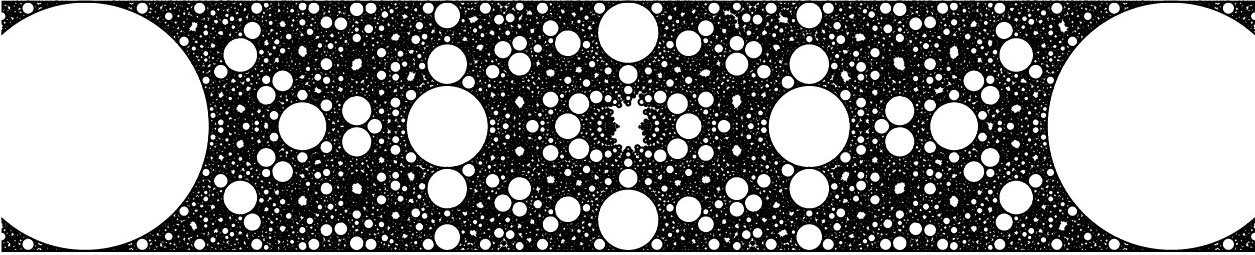}
\includegraphics[width=301pt]{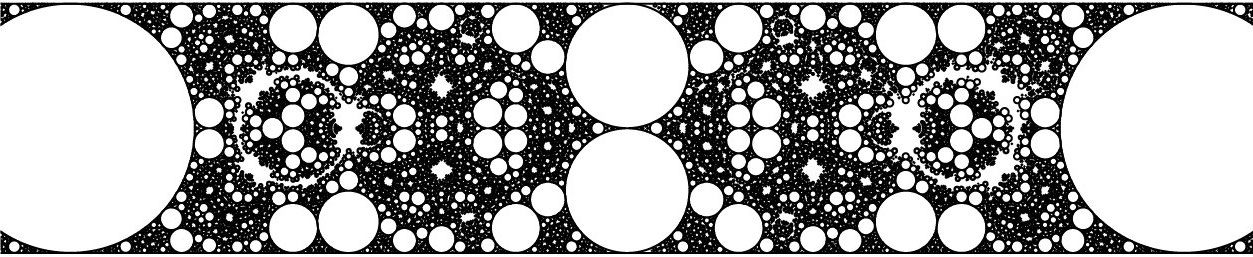}
\includegraphics[width=301pt]{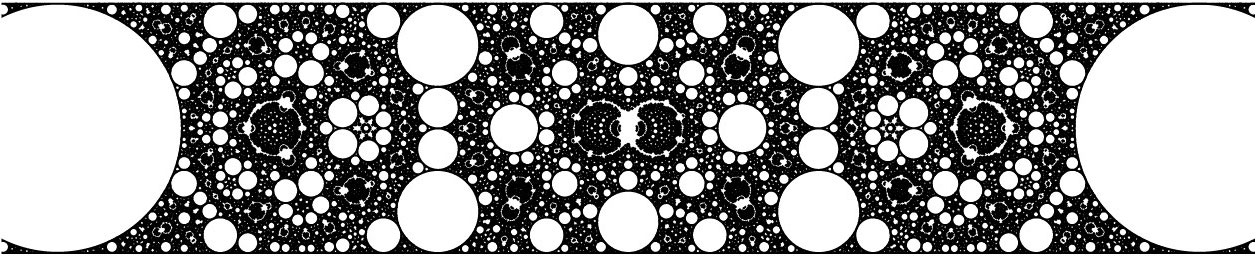}
\includegraphics[width=301pt]{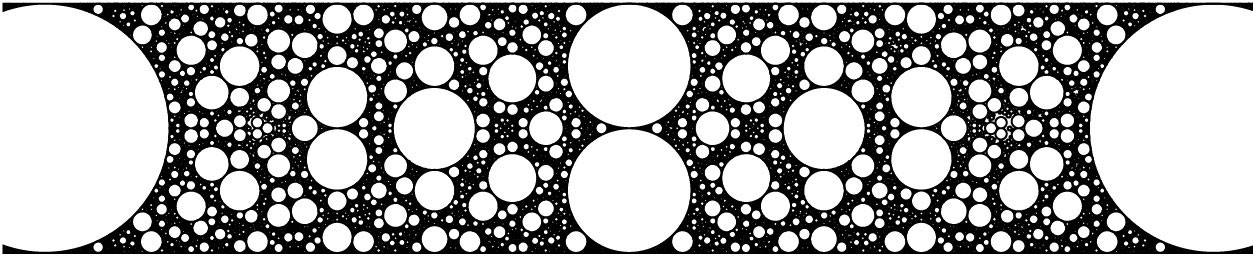}
\includegraphics[width=301pt]{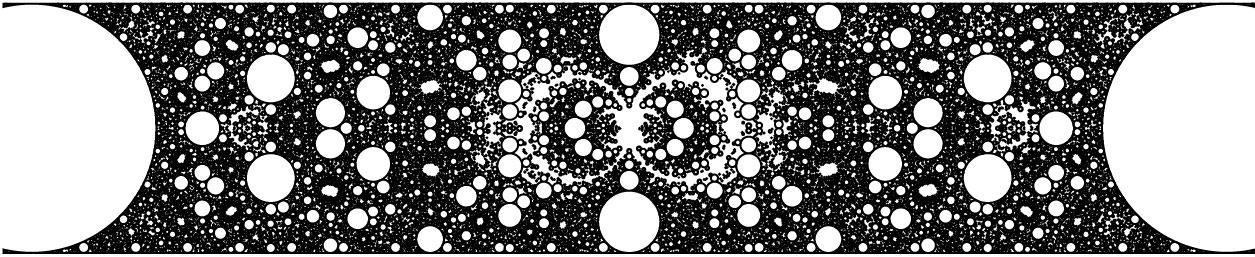}
\includegraphics[width=301pt]{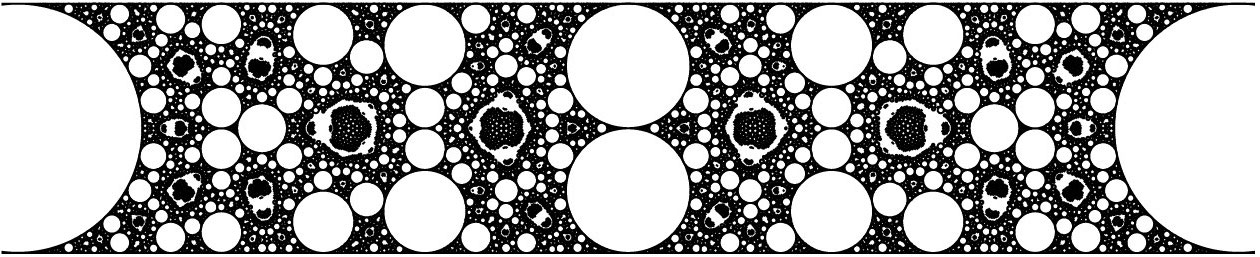}
\includegraphics[width=301pt]{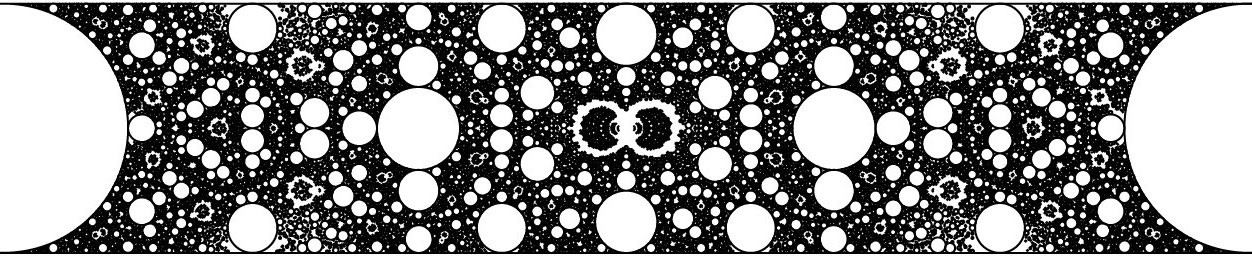}
\includegraphics[width=301pt]{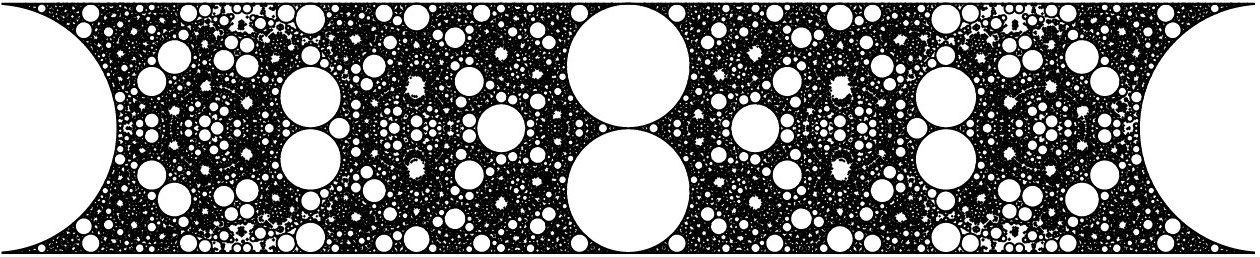}
\caption{\label{fig19to26}  The packings for $n=19$ through $26$. } 
\end{figure}

In Figures \ref{fig1to9}, \ref{fig10to18} and \ref{fig19to26}, we present strip versions of the circle packings for $n=1$ through $26$.  

In Tables \ref{tab1}, \ref{tab2} and \ref{tab3} we give the generators for $\CC_n$ for $n\leq 26$.  For all $n$, the set of generators for $\CC_n$ include the first four generators in Table \ref{tab1}, namely $R_h$, $R_{v_1}$, $R_{v_2}$, and $R_{s_0}$.  The packing is $\CC_n(\vc e_1)$.  Though $\CC_n$ acts transitively on the packing for $n\leq 26$, we have no reason to believe that it does so in general.    

In Tables \ref{tab1}, \ref{tab2} and \ref{tab3}, the types are:
\begin{enumerate}
\item Reflections $R_{\vc n}$ where $\vc n\cdot \vc n<0$, which in $\Bbb C$ correspond to inversion in a circle centered at the given point and with the given radius $r$.
\item  Rotations $\phi_{A,B}$ by $\pi$ about a line in $\Bbb H^3$ with endpoints $A$ and $B$, where $A\cdot A=B\cdot B=0$; see Equation (\ref{mob}) for the corresponding M\"obius map.
\item Inversion $-R_P$ through a point $P$ in $\Bbb H^3$ (so $P\cdot P>0$), which in $\Bbb C$ is inversion in the given circle composed with rotation by $\pi$ about its center.
\end{enumerate}  
The coordinates in $\Bbb C$ correspond to the choice of $\vc e_1$ as the real axis, $v_1$ as the imaginary axis, and $P_3$ 
as the point $i$.

\ignore{
The types of generators are: 
\begin{enumerate}  
\item The reflections
\[
R_{\vc n}(\vc x)=\vc x-\frac{\vc x\ldot \vc n}{\vc n\ldot \vc n}\vc n,
\]
which in $\Bbb C\equiv \partial \H_E$ is represented by inversion in a circle with given center and radius $r$ (except for $h$, $v_1$, and $v_2$, which are lines).  
\item The Bertini involutions 
\[
\phi_{P,Q}(\vc x)= \frac{2((P\cdot \vc x)Q+(Q\cdot \vc x)P)}{P\cdot Q}-\vc x,
\]
which in $\Bbb C$ is the fractional linear transformation
\[
\cc=\mymatrix{ P+Q & -2PQ \\ 2&-(P+Q)}.
\]
Here we are using $P$ and $Q$ to represent both points in $\Bbb R^{3,1}$, and in $\Bbb C$. 
\end{enumerate}
}

\begin{table}
\begin{tabular}{|c|c|c|c|c|} \hline
 $n$ & Type & in $\LL_n$ & in $\Bbb C$ \\
\hline
All $n$& $R_h$ & $h=[-1,1,0,0]$ & $y=1$  \\
 & $R_{v_1}$ & $v_1=[n,n,1,-1]$ & $x=0$  \\
 & $R_{v_2}$ & $v_2=[0,0,-1,1]$ & $x=\sqrt n$  \\
 & $R_{s_0}$ & $s_0=[0,-1,1,0]$ & $0$,  $r=2$ \\
 4 & $R_{s_1}$ & $s_1=[-2,-2,1,1]$ & $2+i$, $r=1$ \\
 5 & $R_{s_1}$ & $s_1=[-1,-3,1,1]$ & $\sqrt 5$, $r=1$ \\
 & $R_{s_2}$ & $s_2=[-5,-5,3,2]$ & $\frac{4}{\sqrt 5}+i $, $r=\frac{1}{\sqrt 5}$ \\
 6 & $R_{s_1}$ & $s_1=[-2,-4,1,1]$ & $\sqrt 6$, $r=\sqrt 2$ \\
 7 & $R_{s_1}$ & $s_1=[-2,-4,1,1]$ & $\sqrt 7$, $r=1$ \\
 & $\phi_{Q_0,Q_1}$ & $Q_0=[-3,-3,1,1]$ & $\sqrt 7+i$ \\
 & &  $Q_1=[-3,-5,3,1]$ & $\frac{\sqrt 7+i}2$ \\
 $8$ & $R_{s_1}$ & $s_1=[-4,-4,1,1]$ & $2\sqrt 2+i$, $r=1$ \\
 & $R_{s_2}$ & $s_2=[-8,-16,5,3]$ & $\frac{3}{\sqrt 2}$, $r=\frac{1}{\sqrt 2}$ \\
 & $R_{s_3}$ & $s_3=[-24,-24,11,5]$ & $\sqrt 3+i$, $r=\frac{1}{\sqrt 3}$ \\
 $9$& $R_{s_1}$ &$s_1=[-3,-5,1,1]$ & $3$, $r=1$ \\
 & $R_{s_2}$ & $s_2=[-18,-18, -5,-4]$ & $\frac{8}{3}+i$, $r=\frac 13$ \\
 & $R_{s_3}$ & $s_3=[-5,-6,2,1]$ & $2+\frac 23 i$, $r=\frac 23$ \\
 $10$ & $R_{s_1}$ & $s_1=[-4,-6,1,1]$ & $\sqrt{10}$, $r=\sqrt 2$ \\
 & $-R_{P}$ & $P=[6,6,-3,-1]$ & $\frac{\sqrt {10}}2+i$, $r=\frac 1{\sqrt 2}$ \\
 $11$ & $\phi_{Q_1,Q_0}$ & $Q_1=[-6,-7,2,1]$ & $\frac{2\sqrt{11}+2i}3$ \\
 & & $Q_0=[-5,-5,1,1]$ & $\sqrt{11}+i$ \\
 & $\phi_{Q_1,Q_2}$ & $Q_2=[-7,-6,2,1]$ & $\frac{2\sqrt{11}+4i}{3}$ \\
 & $\phi_{Q_1,Q_3}$ & $Q_3=[-8,-7,4,1]$ & $\frac{2\sqrt{11}+6i}{5}$ \\
 $12$ & $R_{s_1}$ & $s_1=[-6,-6,1,1]$ & $\sqrt{12}+i$, $r=1$ \\
 & $R_{s_2}$ & $s_2=[-6,-9,2,1]$ & $\frac{2\sqrt{12}}3$, $r=\frac{2}{\sqrt{3}}$ \\
 $13$ & $R_{s_1}$ & $s_1=[-5,-7,1,1]$ & $\sqrt{13}$, $r=1$ \\
 & $R_{s_2}$ & $s_2=[-39,-39,7,6]$ & $\frac{12}{\sqrt{13}+i}$, $r=\frac{1}{\sqrt{13}}$ \\
 & $-R_P$ & $P=[-7,-8,2,1]$ & $\frac{2\sqrt{13}+2i}{3}$, $r=\frac 2 3$ \\
 $14$ & $R_{s_1}$ & $s_1=[-6,-8,1,1]$ & $\sqrt{14}$, $r=\sqrt 2$ \\
 & $R_{s_2}$ & $s_2=[-14,-21,5,2]$ & $\frac{4\sqrt{14}}7$, $r=\frac{2}{\sqrt 7}$ \\
 & $R_{s_3}$ & $s_3=[-10,-10,3,1]$ & $\frac{\sqrt{14}}2+i$, $r=\frac{1}{\sqrt 2}$ \\
 & $R_{s_4}$ & $s_4=[-42,-42,9,5]$ & $\frac{5\sqrt{14}}7+i$, $r=\frac{1}{2\sqrt 2}$ \\
$15$ & $\phi_{Q_0,Q_1}$ & $Q_0=[-7,-7,1,1]$ & $\sqrt{15}+i $ \\
&  & $Q_1=[-25,-27,5,3]$ & $\frac{3\sqrt{15}+3i}4$ \\
& $\phi_{Q_2,Q_3}$ & $Q_2,Q_3=[-6\pm 2\sqrt 3,-9,2,1]$ & $\frac{-\sqrt{15}-(3\pm 2\sqrt 3)i}3$ \\
$16$ & $R_{s_1}$ & $s_1=[-8,-8,1,1]$ & $4+i$, $r=1$ \\
& $R_{s_2}$ & $s_2=[-48,-64,9,7]$ & $\frac 72$ $r=\frac 12 $ \\
& $R_{s_3}$ & $s_3=[-80,-80,25,7]$ & $\frac 7 4 +i$, $r=\frac 14 $ \\
& $-R_P$ & $P=[-9,-10,2,1]$& $\frac{8+2i} 3$, $r=\frac 23$ \\
 \hline
 \end{tabular}
 \caption{\label{tab1}  The generators of $\CC_n$.  }
 \end{table}

\begin{table}
\begin{tabular}{|c|c|c|c|c|} \hline
 $n$ & Type & in $\LL_n$ & in $\Bbb C$ \\
\hline
$17$ & $R_{s_1}$ & $s_1=[-7,-9,1,1]$ & $\sqrt{17}$, $r=1$ \\
& $R_{s_2}$ & $s_2=[-68,-68,9,8]$ & $\frac{16\sqrt{17}}{17}+i$, $r=\frac{1}{\sqrt{17}}$ \\
& $\phi_{Q_1,Q_2}$ & $Q_1=[-10,-11,2,1]$ & $\frac{2\sqrt{17}+2i}{3}$ \\
& & $Q_2=[-11,-10,2,1]$ & $\frac{2\sqrt{17}+4i}3$ \\
& $\phi_{Q_1,Q_3}$ & $Q_3=[-6,-11,2,1]$ & $\frac{2\sqrt{17}-2i}3$ \\
& $\phi_{Q_1,Q_4}$ & $Q_4=[-24,-29,4,3]$ & $\frac{6\sqrt{17}+2i}7$ \\
$18$ & $R_{s_1}$ & $s_1=[-8,-10,1,1]$ & $\sqrt{18}$, $r=\sqrt 2$\\
& $R_{s_2}$ & $s_2=[-10,-14,3,1]$ & $\frac{\sqrt{18}}{2}$, $r=\frac{1}{\sqrt 2}$ \\
& $\phi_{Q_1,Q_2}$ & $Q_1=[-12+\sqrt{2},-12-\sqrt 2,3,1]$ & $\frac{\sqrt{18} + (2-\sqrt 2)i}2$ \\
& & $Q_2=[-12-\sqrt{2},-12+\sqrt 2,3,1]$ & $\frac{\sqrt{18} + (2+\sqrt 2)i}2$ \\
$19$ & $-R_P$ & $P=[-11,-12,2,1]$ & $\frac{2\sqrt{19}+2i}{3}$, $r=\frac{2}3$ \\
&$\phi_{Q_1,Q_0}$ & $Q_1=[-21,-22,3,2]$ & $\frac{4\sqrt{19}+4i}5$ \\
&& $Q_0=[-9,-9,1,1]$ & $\sqrt{19}+i$ \\
& $\phi_{Q_1,Q_2}$ & $Q_2=[-22,-21,3,2]$ & $\frac{4\sqrt{19}+6i}5$ \\
& $\phi_{Q_1,Q_3}$ & $Q_3=[-69,-70,11,6]$ & $\frac{12\sqrt{19}+16i}{17}$ \\
$20$ &$R_{s_1}$& $s_1=[-10,-10,1,1]$ & $\sqrt{20}+i$, $r=1$ \\
& $R_{s_2}$& $s_2=[-20,-25,3,2]$ & $\frac{4\sqrt{20}}5$, $r=\frac{2}{\sqrt 5}$ \\
&$R_{s_3}$& $s_3=[-12,-16,3,1]$ & $\sqrt 5$, $r=1$ \\
& $R_{s_4}$ & $s_4=[-36,-36,5,3]$ & $\frac{3\sqrt{5}}2+i$, $r=\frac 12$ \\
& $R_{s_5}$& $s_5=[-40,-40,7,3]$ & $\frac{3\sqrt{20}}5+i$, $r=\frac{1}{\sqrt 5}$ \\
& $R_{s_6}$ & $s_6=[-15,-15,4,1]$ & $\frac{2\sqrt{20}}5+i$, $r=\frac{1}{\sqrt 5}$ \\
$21$ & $R_{s_1}$ & $[-9,-11,1,1]$ & $\sqrt{21}$, $r=1$ \\
& $R_{s_2}$ & $[-105,-105,11,10]$ & $\frac{20\sqrt{21}}{21}$, $r=\frac{1}{\sqrt{21}} $ \\
& $R_{s_3}$ & $[-12,-15,2,1]$ & $\frac{2\sqrt{21}}3$, $r=\frac{2}{\sqrt 3}$ \\
& (see text) & $T$ & $\ss$  \\
$22$ & $R_{s_1}$ & $s_1=[-10,-12,1,1]$ & $\sqrt{22}$, $r=\sqrt 2$ \\
& $-R_P$ & $P=[-13,-14,2,1]$ & $\frac{2\sqrt{22}+2i}3$, $r=\frac 23$ \\
& $\phi_{Q_1,Q_2}$ & $Q_1,Q_2=[-12\pm 2\sqrt 2,-15,3,1]$ & $\frac{\sqrt{22}\pm i\sqrt 2}2$ \\
$23$ & $\phi_{Q_0,Q_1}$ & $Q_0=[-11,-11,1,1]$ & $\sqrt{23}+i$ \\
& & $Q_1=[-63,-65,7,5]$ & $\frac{5\sqrt{23}+5i}{6}$ \\
& $\phi_{Q_2,Q_3}$ & $Q_2=[-14,-15,2,1]$ & $\frac{2\sqrt{23}+2i}3$ \\
& & $Q_3=[-15,-17,3,1]$ & $\frac{\sqrt{23}+i}2$ \\
& $\phi_{Q_2,Q_4}$ & $Q_4=[-15,-14,2,1]$ & $\frac{2\sqrt{23}+4i}3$ \\
& $\phi_{Q_2,Q_5}$ & $Q_5=[-10,-15,2,1]$ & $\frac{2\sqrt{23}-2i}{3}$ \\
& $\phi_{Q_3,Q_6}$ & $Q_6=[-17,-15,3,1]$ & $\frac{\sqrt {23}+3i}2$ \\
& $\phi_{Q_3,Q_7}$ & $Q_7=[-11,-17,3,1]$ & $\frac{\sqrt{23}-i}2$ \\
 \hline
 \end{tabular}
 \caption{\label{tab2}  The generators of $\CC_n$, continued.  }
 \end{table}

\begin{table}
\begin{tabular}{|c|c|c|c|c|} \hline
 $n$ & Type & in $\LL_n$ & in $\Bbb C$ \\
\hline
$24$ & $R_{s_1}$ & $s_1=[-12,-12,1,1]$ & $\sqrt{24}+i$, $r=1$ \\
& $R_{s_2}$ & $s_2=[-120,-144,11,13]$ & $\frac{11\sqrt{24}}{12}$, $r=\frac{1}{\sqrt 6}$ \\
& $R_{s_3}$ & $s_3=[-264,-264,29,19]$ & $\frac{19\sqrt{24}}{24}+i$, $r=\frac{1}{2\sqrt 6}$ \\
& $\phi_{Q_1,Q_2}$ & $Q_1,Q_2=[-12\pm 2\sqrt 3,-15,2,1]$ & $\frac{4\sqrt 6\pm 2\sqrt 3 i}3$ \\
& $\phi_{Q_3,Q_4}$ & $Q_3,Q_4=[-18\pm \sqrt{3},-18,5\pm \sqrt{3},1]$ & $\frac{8\sqrt 6\mp 4\sqrt 2+(10\pm 2\sqrt 3)i}{11}$ \\
$25$ & $R_{s_1}$ & $s_1=[-11,-13,1,1]$ & $5$, $r=1$ \\
& $R_{s_2}$ & $s_2=[-150,-150,13,12]$ & $\frac{24}5+i$, $r=\frac 15$ \\
& $-R_{P_1}$ & $P_1=[-15,-16,2,1]$ & $\frac{10+2i}3$, $r=\frac 23$ \\
& $-R_{P_2}$ & $P_2=[-28,-29,3,2]$ & $4+\frac 4 5 i$, $r=\frac 25 $ \\
& $-R_{P_3}$ & $P_3=[-18,-19,4,1]$ & $2+\frac 45 i$, $r=\frac 25 $ \\
$26$ & $R_{s_1}$ & $s_1=[-12,-14,1,1]$ & $\sqrt{26}$, $r=\sqrt 2$ \\
& $\phi_{Q_1,Q_2}$ & $Q_1=[-16,-17,2,1]$ & $\frac{2\sqrt{26}+2i}3$ \\
&& $Q_2=[-17,-16,2,1]$ & $\frac{2\sqrt{26}+4i}3$ \\
&$\phi_{Q_1,Q_3}$ & $Q_3=[-12,-17,2,1]$ & $\frac{2\sqrt{26}-2i}3$ \\
& $\phi_{Q_1,Q_4}$ & $Q_4=[-28,-31,3,2]$ & $\frac{4\sqrt{26}+2i}5$ \\
& $\phi_{Q_5,Q_6}$ & $Q_5,Q_6=[-18\mp \sqrt 2, -18\pm \sqrt 2,3,1]$ & $\frac{\sqrt{26}+(2\pm \sqrt 2)i}2$ \\
 \hline
 \end{tabular}
 \caption{\label{tab3}  The generators of $\CC_n$, continued.  }
 \end{table}

\end{document}